\documentclass[review,a4paper,final]{elsarticle}
\makeatletter
\def\ps@pprintTitle{%
 \let\@oddhead\@empty
 \let\@evenhead\@empty
 \def\@oddfoot{\centerline{\thepage}}%
 \let\@evenfoot\@oddfoot}
\makeatother
\newif\ifusetikzandpgfplots
\newif\ifshowlinenumbers
\usepackage{mathptmx}      % use Times fonts if available on your TeX system
\usepackage{fix-cm}
\usepackage{etex} %some fix for hyperref
\usepackage{ifluatex}
\ifluatex
  \usepackage{lualatex-math}
\else
  \usepackage[utf8]{inputenx}
  \input{ix-utf8enc.dfu}
\fi
 \usepackage{microtype}
\usepackage[dvipsnames,svgnames,table]{xcolor} % Farben

\usepackage[T1]{fontenc} %Kodierung von Zeichensaetzen. 
\usepackage{latexsym}
\usepackage[full]{textcomp}
\usepackage{stmaryrd}
\usepackage[a4paper]{geometry}
\usepackage[shortlabels]{enumitem}
\usepackage{amsfonts}%Fonts = Schriftarten der American mathematical Society (ams)
\usepackage{amssymb}%Mathematische Symbole der ams.
\usepackage{bbm}
\usepackage{bm} % boldmath
\usepackage{amsmath}
\usepackage{exscale}
\usepackage{amstext}
\usepackage{amsthm}
 
\usepackage{float,scrhack}
\usepackage{wrapfig}
\usepackage[textfont=it,labelfont=bf,format=plain]{caption}
\usepackage{subcaption}

\usepackage{hhline} %stellt den \hhline in Tabellen zur Verfuegung
\ifusetikzandpgfplots
  \usepackage{tikz}
  \usepackage{pgfplots}
  \pgfplotsset{compat=newest,%
	      every axis plot post/.append style={
		every mark/.append style={scale=0.8,solid},
	      },
	      every axis plot/.append style={thick},
  }
  \usetikzlibrary{arrows,%
		  arrows.meta,%
		  petri,%
		  topaths,%
		  fit,%
		  positioning,%
		  decorations.pathmorphing,%
		  backgrounds,
		  calc,%
  }%

  \usetikzlibrary{external}
  \usepgfplotslibrary{external}
  \tikzset{external/system call={lualatex \tikzexternalcheckshellescape --shell-escape -halt-on-error -interaction=batchmode -jobname "\image" "\texsource"}}

  \tikzset{
    png export/.style={
      external/system call=%
      {lualatex \tikzexternalcheckshellescape --shell-escape -halt-on-error -interaction=batchmode -jobname %
      "\image" "\texsource" && convert -density 600 "\image.pdf" "\image.png"},
    }
  }

  \tikzset{%
    /pgf/images/external info,
    use png/.style={png export,png import},
    png import/.code={%
      \tikzset{%
	/pgf/images/include external/.code={%
	  \includegraphics%
	  [width=\pgfexternalwidth,height=\pgfexternalheight]%
	  {{##1}.png}%
	}%
      }%
    }%
  }

  \tikzset{png export}
  \tikzifexternalizing{%
  }{%
  \usepackage{pdfpages}
  }%
\fi

\usepackage{nameref}
\usepackage[english]{varioref}%Stellt erweiterte Befehle zur Referenzierung von
\usepackage{hyperref}
\hypersetup{
      unicode=true,           % non-Latin characters in Acrobat’s bookmarks
     plainpages=false,		% do page number anchors as plain arabic
     linkcolor=blue,          % color of internal links (change box color with linkbordercolor)
     citecolor=darkgreen,        % color of links to bibliography
     filecolor=black,      % color of file links
     urlcolor=blue           % color of external links
     pdfborder={0 0 1},		% width of pdf link border 0 0 1, 0 0 0 = colorlinks 
     linkbordercolor=gray!15,
     citebordercolor=green!15,
}
\usepackage{bookmark}
\ifshowlinenumbers
  \usepackage{lineno}
  \modulolinenumbers[5]
\fi

\newcommand{\dom}{\Omega}

\newcommand{\vt}[1]{\bm{#1}}
\newcommand{\oB}[1]{\partial\Omega^{#1}}
\newcommand{\plni}[2][l]{p_{#1}^{n,#2}}
\newcommand{\Splni}[1]{S_l(p_l^{n,#1})}

\newcommand{\gradPlniPlusGravity}[1]{\bm{\nabla} \bigl( p_l^{n,#1} + z \bigr)}
\newcommand{\klni}[2][l]{k_#1\bigl(S_#1(p_#1^{n,#2})\bigr)}
\newcommand{\kln}[1]{k_{#1}\bigl(S_{#1}(p_{#1}^{n})\bigr)}

\newcommand{\gli}[2][l]{g_{#1}^{#2}}
\newcommand{\epli}[2][l]{e_{p,#1}^{#2}}
\newcommand{\epliOnGamma}[2][l]{e_{p,#1}^{#2}}%{|_\Gamma}
\newcommand{\gradepli}[2][l]{\bm{\nabla} e_{p,#1}^{#2}}
\newcommand{\egli}[2][l]{e_{g,#1}^{#2}}
\newcommand{\flux}[2][l]{\vt{F_{#1}^{n,#2}}}
\newcommand{\Fl}[1]{\vt{F_{#1}}}
\newcommand{\Fln}[2][n]{\vt{F^{#1}_{#2}}}
\newcommand{\spl}{\bigl\langle}
\newcommand{\spr}{\bigr\rangle}
\newcommand{\tspl}{\langle}
\newcommand{\tspr}{\rangle}
\newcommand{\bspl}{\Bigl\langle}
\newcommand{\bspr}{\Bigr\rangle}
\newcommand{\completion}[2]{\overline{#1}{}^{#2}}

\newcommand{\Fs}{\mathcal{V}}
\newcommand{\Tracespace}{H^{1/2}_{00}(\Gamma)}
\newcommand{\tGamma}{{\Gamma}}
\newcommand{\onGamma}[1]{{#1}{|_\Gamma}}
\newcommand{\restrictTo}[2]{{#1}{}{|_{#2}}}
\newcommand{\RR}{\mathbb{R}}
\newcommand{\RRd}{\mathbb{R}^d}
\newcommand{\NN}{\mathbb{N}}

\DeclareMathOperator{\dv}{\bm{\nabla}\cdot }

\providecommand{\norm}[1]{\lVert#1\rVert}

\definecolor{olivedrab}{RGB}{107,142,35}
\newif\ifoptionaltext
\newcommand{\optionaltextcolor}{\color{gray}}
\newif\ifcomments
\commentstrue % comment out to hide optional text

\newtheorem{theorem}{Theorem}
\newtheorem{lemma}[theorem]{Lemma}
\newdefinition{notation}{Notation}
\newdefinition{problem}{Problem}
\newdefinition{assumptions}{Assumptions}
\newdefinition{remark}{Remark}

\graphicspath{{./Figures/}}

\begin{document}
\begin{frontmatter}
%%% ARTICLE TITLE
\title{A linear domain decomposition method for partially saturated flow in porous media}
% \title{Elsevier \LaTeX\ template\tnoteref{mytitlenote}}
% \tnotetext[mytitlenote]{Fully documented templates are available in the elsarticle package on \href{http://www.ctan.org/tex-archive/macros/latex/contrib/elsarticle}{CTAN}.}

%%% AUTHORS AND AFFILIATIONS
%%%% affiliation file

\author[davidaddress]{David Seus\corref{correspondingAuthor}}%\fnref{myfootnote}}
\cortext[correspondingAuthor]{Corresponding author:}
\ead{david.seus@ians.uni-stuttgart.de}
\address[davidaddress]{
  Institute of Applied Analysis and Numerical Simulation, %\\ 
  Chair of Applied Mathematics, %\\
  Pfaffenwaldring 57, %\\
  70569 Stuttgart, %\\
  Germany
%               %Tel.: +49 711 685 61994\\
%               %Fax:  +49 711 685-65599\\
}


\author[koondimainaddress,sorinaddress]{Koondanibha Mitra}
% \ead{k.mitra@tue.nl}
% \author[mysecondaryaddress]{Global Customer Service\corref{mycorrespondingauthor}}
% \cortext[mycorrespondingauthor]{Corresponding author}

\address[koondimainaddress]{%MF 7.123 \\
  Department of Mathematics and Computer Science, % \\
  Technische Universiteit Eindhoven, % \\
  PO Box 513, % \\
  5600 MB  Eindhoven, % \\
  The Netherlands
}
% \address[mysecondaryaddress]{360 Park Avenue South, New York}

\author[sorinaddress,florinaddress]{Iuliu Sorin Pop}
%\ead{sorin.pop@uhasselt.be}
\address[sorinaddress]{
  Faculty of Sciences, %\\
  Hasselt University, %\\
  Campus Diepenbeek, %\\
  Agoralaan Building D, %\\
  BE3590 Diepenbeek, %\\
  Belgium
}
\author[florinaddress]{Florin Adrian Radu}
%\ead{florin.radu@uib.no}
\address[florinaddress]{
  Department of Mathematics, %\\
  University of Bergen, %\\
  P. O. Box 7800, %\\
  N-5020 Bergen, %\\
  Norway
}
\author[davidaddress]{Christian Rohde}
%\ead{christian.rohde@ians.uni-stuttgart.de}

\begin{abstract}
  The Richards equation is a nonlinear parabolic equation that is commonly used for modelling saturated/unsaturated flow in porous media. We assume that the medium occupies a bounded Lipschitz domain partitioned into two disjoint subdomains separated by a fixed interface $\Gamma$. 
  This leads to two problems defined on the subdomains which are coupled through conditions expressing flux and pressure continuity at $\Gamma$. After an Euler implicit discretisation of the resulting nonlinear subproblems a linear iterative ($L$-type) domain decomposition scheme is proposed. The convergence of the scheme is proved rigorously. 
  In the last part we present numerical results that are in line with the theoretical finding, in particular the unconditional convergence of the scheme. We further compare the scheme to other approaches not making use of a domain decomposition. 
  Namely, we compare to a Newton and a Picard scheme. We show that the proposed scheme is more stable than the Newton scheme 
  while remaining comparable in computational time, even if no parallelisation is being adopted. 
  Finally we present a parametric study that can be used to optimize the proposed scheme.
\end{abstract}


\end{frontmatter}

\ifshowlinenumbers
  \linenumbers
\fi

\section{Introduction}
Unsaturated flow processes through porous media appear in a variety of physical situations and applications. 
Notable examples are soil remediation, enhanced oil recovery, $CO_2$ storage, harvesting of geothermal energy, or the design of filters and fuel cells. 
Mathematical modelling and numerical simulation are essential for understanding such processes, since measurements and experiments are very difficult if not impossible, and hence only limitedly
available. 
The associated mathematical and computational challenges are manifold. The mathematical models are usually coupled systems
of nonlinear partial differential equations and ordinary ones, involving largely varying physical properties and parameters, like porosity,
permeability or soil composition. Together with the large scale and possible complexity of the domain, this poses significant computational challenges, making the design and analysis of robust discretisation methods a non-trivial task.

In this work we focus on saturated/unsaturated flow of one fluid (water) in a porous medium (e.g. the subsurface) occupying the domain $\dom \subset \RRd$ ($d \in \{1, 2, 3\}$). Besides water, a second phase (air) is present,  which is assumed to be at a constant (atmospheric) pressure. This situation is described by the Richards equation, here in  pressure formulation
\begin{align}
  \Phi\partial_t S(p) - \nabla\cdot \left[ \frac{\vt{K}}{\mu}k_r\bigl( S(p) \bigr) \bm{\nabla}\bigl( p +
z\bigr) \right]  &= 0, \label{RichardsOriginal}
\end{align}
see e.g. \cite{Helmig97}, originally \cite{Richardson1922,Richards1931}. In the above $\Phi$ denotes the porosity, $S$ is the water saturation, $p$ is
the water pressure, $k_r$ is
the relative permeability, $\vt{K}$ the intrinsic permeability and $z = -\rho_wg x_3$ is the
gravitational term in direction of the $x_3$-axis. Finally, $g$ is the gravitational acceleration, $\rho_w$ the water density and $\mu$ its viscosity. With $T > 0$ being a maximal time, the equation is defined for the time $t \in (0, T)$ on the bounded Lipschitz domain $\dom$. 
% The boundary and initial conditions will be mentioned later.

Below we propose a domain decomposition (DD) scheme for the numerical solution of \eqref{RichardsOriginal}. 
To this aim we assume that $\dom$ is partitioned into two subdomains $\dom_l$ ($l \in \{1, 2\} $) separated by a Lipschitz-continuous interface $\Gamma$, see Fig. \ref{dddrawing}. 
In other words one has $\dom = {\dom_1} \cup {\dom_2} \cup \Gamma$. The restriction to two subdomains is made for the ease of presentation, but the scheme can be extended straightforwardly to more subdomains. 
In each $\dom_l$ ($l \in \{1, 2\}$) we use the physical pressure $p_l$ as primary variable. 
Furthermore, the permeability and porosity in each of the subdomains may be different and even discontinuous, which is the case of a heterogeneous medium consisting of block-type heterogeneity (like a fractured medium).

In view of its relevance for manifold applications in real life, Richards equation has been studied extensively, both analytically and numerically, and the dedicated literature is extremely rich. 
We restrict ourselves here by mentioning \cite{AltandLuckhaus,AltLuckhausVisintin1984} for the existence of weak solutions and \cite{Otto1996} for the uniqueness.  %An overview of the state of the art of numerical treatment is presented in \cite{List2016} as well as in \cite{Berninger2009}.
Numerical schemes for the Richards equation, or in general for degenerate parabolic equations, are analysed in
\cite{Arbogast1996,Nochetto1988,Pop2002,Radu2004,Radu2008,Radu2017,Klausen2008,Yotov1997,Woodward2000}. 
Most of the papers are considering the backward Euler method for the time discretisation in view of the low regularity of the solution, see \cite{AltandLuckhaus}, and to avoid restrictions on the time step size.

%\marginpar{There is no FEM in \cite{Pop2002} or DG in \cite{Berninger2011}!}
Different approaches with regard to spatial discretisation have been considered. Galerkin finite elements were used in
\cite{Nochetto1988,Slodicka2002,Berninger2011}. Discontinuous Galerkin finite element schemes for flows through (heterogeneous) porous media have been studied in \cite{Bastian2007,Epshteyn}.
Finite volume schemes including multipoint flux approximation ones for the Richards equation are analysed in \cite{Eymard1999,Eymard2006,Klausen2008}, and mixed finite elements in \cite{Arbogast1996,Bause2004,Radu2004,Radu2008,Radu2017,Woodward2000,Yotov1997}. Such schemes are locally mass conservative.

Applying the Kirchhoff transformation \cite{AltandLuckhaus} brings the mathematical model to a form that simplifies  mathematical and numerical analysis, see e.g.
\cite{Nochetto1988,Arbogast1996,Radu2004,Radu2008}. 
However, the transformed unknown is not directly related to a physical quantity like the pressure, and therefore a postprocessing step is required after a numerical approximation of the solution has been obtained. 
Alternatively, one may develop numerical schemes for the original formulation and in terms of the physical quantities. %, see e.g. \cite{michel}.
Nevertheless, when proving the convergence rigorously, one often resorts to a Kirchhoff transformed formulation as intermediate step. Alternatively, sufficient regularity of the solution, e.g. by avoiding cases where the medium is completely saturated, or completely dry, has to be assumed. 
We point out that in this work we will not make use of the Kirchhoff transformation, keeping the equation in a more relevant form for applications.
%
%Alternatively, convergence analysis of a mixed FEM discretisation as well as FEM and Finite Difference
%approaches  by restricting generality to the strictly unsaturated case are performed in \cite{Arbogast1993,Radu2014}.
%

%\marginpar{No Newton in \cite{Radu2008}}
If implicit methods are adopted for the time discretisation, the (elliptic or fully discrete) problems obtained at each time step are nonlinear. For solving these, different
approaches have been proposed. 
Examples are the Newton method  \cite{Bergamaschi1999,Park1995,Brenner}, the Picard/modified Picard method \cite{Celia1990,Lott2012}, or the J\"ager-Kacur method \cite{JaegerKacur1995,Kacur1999}. 
We refer to \cite{Radu2006} for the convergence analysis of such nonlinear schemes. 
Assuming that the initial guess is the solution from the previous time step, the convergence of such schemes can only be guaranteed under severe restriction for the time step in terms of the mesh size. 
Additionally, % convergence is obtained only after having
regularizing the problem is required, which prevents the Jacobian from becoming singular.
% , and if the initial guess is close enough to the solution, which induces a severe constraint on the time step.
%\marginpar{L-scheme is not a modified Newton one but rather a fixed point method}
Such difficulties do not appear when the L-scheme is being used, which is a fixed point scheme transforming the iteration into a contraction, \cite{Pop2004365,Radu2015134,Slodicka2002}. 
The convergence is merely linear
%, but in the $H^1$ norm or in its discrete $H^1$-counterpart,
but in a better norm ($H^1$) and requires no regularization or  severe constraint on the time step. We also refer to \cite{List2016} for a combination of the Newton method and the L-scheme. 
Moreover, we mention \cite{Radu2017} for the application of the L-scheme to H\"older instead of Lipschitz continuous nonlinearities.

Independent of the chosen discretisation method and of the linearisation scheme, domain decomposition (DD) methods offer an efficient way to reduce the computational complexity of the problem, and to perform calculations in parallel. 
This is in particular interesting whenever domains with block type heterogeneities are considered, as DD schemes allow decoupling the models defined in different homogeneous subdomains and solving these numerically in parallel. 
We refer to \cite{QuarteroniValli2005} for a detailed discussion of linear DD methods and to \cite{Dolean} for a general introduction into the subject. 
Comprehensive studies of nonlinear DD schemes in the field of fluid dynamics can be found in \cite{Glowinski,Dryja,Tai}. For articles strictly related to porous media flow models, we refer to \cite{Skogestad2013,Skogestad2016} for an overview of different overlapping domain decomposition strategies. 
Linear and nonlinear additive Schwartz methods are compared, and the use of such methods as linear and nonlinear preconditioners is discussed. 
Regardless of the type of the DD scheme, choosing the optimal parameters is a key issue. Such aspects are 
analysed e.g. in \cite{GanderDubois,GanderHalpern}. 
We also refer to \cite{Sarah} for a DD algorithm for porous media flow models, where a-posteriori estimates are used to optimize the parameters and the number of iterations.

Recall that the Richards equation is a nonlinear evolution equation. 
For solving this type of equation, methods like parareal \cite{Gander1} and wave-form relaxation \cite{GanderRohde,GanderKwok} have been proposed. 
The main ideas there are to decompose the problem into separate problems defined in time/space-time domains.
DD methods for the Richards equation are discussed in \cite{Berninger2010,Berninger2015}. 
In these papers the domain is decomposed into multiple layers and the Richards models restricted to adjacent layers are coupled by Robin type boundary conditions. 
The approach uses nonoverlapping domain-decomposition and generalises the ideas of the method introduced in \cite{Lions1988} for linear elliptic problems (see also \cite{QinXu2006,Hoang}), leading to decoupled, nonlinear problems in the subdomains.

Here we consider a linear DD scheme for the numerical approximation of the time discrete problems obtained after substructuring into subproblems and performing an Euler implicit time stepping. 
A nonoverlapping DD scheme (referred to henceforth as LDD scheme) inspired by the DD method introduced in \cite{Lions1988} is defined. 
The LDD iterations are linear, based on an L-type scheme.
This approach differs from the one commonly used when dealing with nonlinear elliptic problems in the context of DD. 
In most cases, the DD iterations lead to nonlinear subproblems. 
For solving these, iterative methods in each subdomain are applied. 
In our approach, the linearisation step is part of the DD iterations, which reduces the computational time.
More precisely, the L-scheme idea is combined with the nonoverlapping DD scheme such that the equations defined in each subdomain along with the Robin type coupling conditions on the interface become  linear. 
For the resulting scheme we prove rigorously the unconditional convergence, and provide numerical examples supporting the theoretical findings and demonstrating its effectiveness.

The paper is structured as follows. In Sec. \ref{SectionProblemFormulationAndIterationScheme} we present the mathematical model and introduce the DD scheme.  
Section \ref{SectionAnalysisOfScheme} contains the analysis of the scheme. 
Finally, Sec. \ref{SectionNumericalExperiments} provides numerical experiments in two spatial dimensions, together with an analysis of the practical performance of the scheme. 
This includes a comprehensive comparison (including robustness and efficiency) between the proposed DD scheme and standard monolithic schemes based on Newton, modified Picard as well as the L-scheme.

% Drawing
\begin{figure}[htbp]
  \centering
  \includegraphics[width=0.7\textwidth]{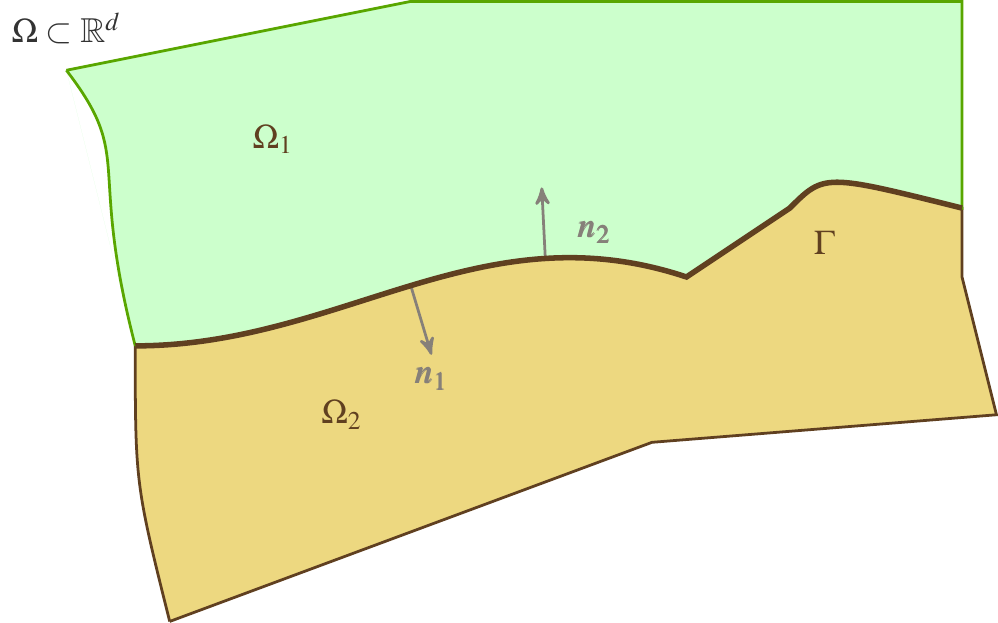}
  \caption{Illustration of the domain $\dom = \dom_1 \cup \,\dom_2 $ $\subset \RRd$ with fixed interface
  $\Gamma$. Also shown are the normal vectors along the interface. \label{dddrawing} }
  % and the outer boundaries $\partial\dom^l$, $l=1,2$
\end{figure}

\section{Problem formulation and iterative scheme} \label{SectionProblemFormulationAndIterationScheme}
\subsection{Problem formulation}
Recall that $T > 0$ and $\dom \subset \RRd$ is a bounded Lipschitz domain partitioned in two subdomains $\dom_{1, 2}$, separated by the Lipschitz-continuous interface
$\Gamma$. 
The boundary of $\dom$ is denoted by $\partial \dom$  and the portions of $\partial \dom$ that are also boundaries of $\dom_l$ are denoted by $\partial\dom_l$ (see also Fig. \ref{dddrawing}). 
To ease the presentation, the two subdomains are assumed to be homogeneous and isotropic, i.e. we can have two different relative permeabilities $k_r = k_{r,l}$ on each  $\dom_l$, the intrinsic permeabilities $\vt{K} = K_l$ are scalar and the two porosities $\Phi_l$ ($l = 1, 2$) are constant. The product $\frac{K_l\, k_{r,l}}{\Phi_l\mu_l}$ in \eqref{RichardsOriginal} is abbreviated by $k_l$ henceforth.
We solve equation \eqref{RichardsOriginal} in $\Omega$ together with initial and homogeneous Dirichlet boundary conditions. We refer to \cite{Berninger2010,Schweizer} for more general conditions, including outflow-type ones. 

On the two subdomains, the problem  transforms into two subproblems, coupled through two conditions at the interface $\Gamma$: the continuity of the normal fluxes and the continuity of the pressures. 
 With the fluxes
$ \Fl{l} := -k_l\bigl(S_l(p_l)\bigr)\bm{ \bm{\nabla} \bigl( } p_l + z \bigr)$, \eqref{RichardsOriginal} becomes
%% Classical Problem formulation
%\begin{problem}[classical formulation] \label{ClassicalFormulation}
    %Richards Eq
\begin{align}
  \partial_t S_l(p_l) + \bm{\nabla}\cdot \Fl{l}
    &= 0
      && \mbox{in } \dom_l\times(0,T], \label{ClassicalFormulation1}\\
  \Fl{1} \cdot \vt{n_1}
    &= -\Fl{2}\cdot \vt{n_2}
      && \mbox{on } \Gamma\times[0,T], \label{ClassicalFormulation2}  \\
  p_1 &= p_2  && \mbox{on } \Gamma\times(0,T], \label{ClassicalFormulation3} \\
  p_l &= 0  && \mbox{on } \partial\dom_l\times(0,T]. \label{ClassicalFormulation4}
\end{align}
This is closed by the initial conditions $p_l(\cdot,0):= p_{l,0}$ in $\dom_l$, where $p_l$ is the water pressure on $\dom_l$, $l=1,2$, and $k_{l}$ are (given) scaled relative permeability functions, that are assumed to be smooth enough. In the above, $\vt{n_l}$ stands for the outer unit normal vector at $\partial\dom_l$.
\subsubsection*{Semi-discrete formulation (discretisation in time)}
For the time discretisation we let $N \in \mathbb{N}$ be a given and $\tau  := \tfrac{T}{N}$ be the corresponding time step. 
Then $p_l^n$ is the approximation of the pressure $p_l$ at time $t^n = n \tau$. The Euler implicit discretisation of
(\ref{ClassicalFormulation1}) -- (\ref{ClassicalFormulation4}) reads
\begin{align}
  S_l\bigl(p_l^n\bigr) -S_l\bigl(p_l^{n-1}\bigr) + \tau \bm{\nabla} \cdot \Fln{l}
    &= 0
      && \hspace{-3mm}\mbox{in } \dom_l, \label{SemiDiscreteFormulation1}\\
  \Fln{1}\cdot \vt{n_1}
    &= -\Fln{2}\cdot \vt{n_2}
      && \hspace{-3mm}\mbox{on } \Gamma, \label{SemiDiscreteFormulation2}  \\
  p_1^n &= p_2^n  && \hspace{-3mm}\mbox{on } \Gamma, \label{SemiDiscreteFormulation3} \\
  p_l^n &= 0  && \hspace{-3mm}\mbox{on } \partial\dom^l,   \label{SemiDiscreteFormulation4}
\end{align}
where $\Fln{l}:= k_l\bigl(S_l(p_l^n)\bigr)\bm{\nabla} \bigl( p_l^n +
z \bigr) $ is the flux at time step $t^n$. Observe that (\ref{SemiDiscreteFormulation2}) and (\ref{SemiDiscreteFormulation3}) are the coupling conditions at the interface $\Gamma$.

\subsection{The LDD iterative scheme} \label{SubsectionIterationScheme}
If $\bigl(p_1^{n-1},p_2^{n-1}\bigr)$ is known, $\bigl(p_1^{n},p_2^{n}\bigr)$ can be obtained by solving the nonlinear system (\ref{SemiDiscreteFormulation1})--(\ref{SemiDiscreteFormulation4}). 
To this end, we define an iterative scheme that
uses Robin type conditions at $\Gamma$ to decouple the subproblems in $\dom_l$, and linearises the terms due to the saturation-pressure dependency by adding stabilisation terms that cancel each other in the limit (see e.g. \cite{List2016,Pop2004365}). 
Specifically, assuming  that for some $i\in\NN$ the approximations $\plni{i-1}$ and
$\gli{i-1}$ are known, we seek $\bigl(\plni[1]{i},\plni[2]{i}\bigr)$ % \in C^{2,1}(\dom_{1T})\times C^{2,1}(\dom_{2T})
solving the problems
\begin{align}
  L_l \plni{i} - L_l \plni{i-1}
   + \tau\bm{\nabla}\cdot\flux{i} %\nonumber \\%
  &= - S_l\bigl(p_l^{n,i-1}\bigr) + S_l\bigl(p_l^{n-1}\bigr)  %
   &&\mbox{ in } \dom_l ,
  \label{strongiterschemeeqn1}  \\
  \flux{i} \cdot \vt{n_l} &= g^i_l + \lambda p_l^{n,i}
   &&\mbox{ on } \Gamma\times[0,T],
  \label{strongiterschemeeqn2} \\[0.9mm]
   g_l^i :\hspace{-1mm} &= -2\lambda p_{3-l}^{n,i-1}-
  g_{3-l}^{i-1}.   \label{stronggliupdate}
\end{align}
Following the previously introduced notation, $\flux{i} :=- k_l\bigl(S_l(p_l^{n,i-1})\bigr) \bm{\nabla} \bigl( p_l^{n,i} + z \bigr) $
denotes the linearised flux at iteration $i$. By $\lambda \in(0,\infty)$, we denote
 a (free to be chosen) parameter used to weight the influence of the pressure on the interface conditions at $\Gamma$. 
The parameters $L_l > 0$ must adhere to some mild constraints in order for the scheme to converge, which will be discussed later, but other than that, are arbitrary. 
The iteration starts with
$$\plni{0}:= p_l^{n-1}, \qquad \text{ and } \qquad \gli{0}:=  \Fln[n-1]{l}\cdot\vt{n_l} -
\lambda p_l^{n-1}, $$
and clearly, the difference $L_l \plni{i} - L_l \plni{i-1}$ is vanishing in case of convergence.

%
% EXPLENATORY REMARK
\begin{remark}\label{rem:idea}%[Idea behind the introduction of $\bm{g_l^i}$]
  The usage of the terms $\gli{i}$ and of the parameter $\lambda$ is motivated by the following. With the notation $f_l^n := \Fln{l} \cdot
  \vt{n_l}$, the transmission conditions \eqref{SemiDiscreteFormulation2}-\eqref{SemiDiscreteFormulation3} become $f_1^n = -f_2^n$ and $p_1^n = p_2^n$. For any $\lambda \neq 0$, these are equivalent to
  %\begin{equation*}
  %  \begin{aligned}
  %    (f_1^n + f_2^n) + \lambda(p_1^n - p_2^n) &= 0, \\
  %    (f_1^n + f_2^n) - \lambda(p_1^n - p_2^n) &= 0.
  %  \end{aligned}
  %\end{equation*}
%This can be reformulated as
  \begin{equation}
    \begin{aligned}
      f_1^n & = (-f_2^n  - \lambda p_2^n) + \lambda p_1^n, \\
      f_2^n & = (-f_1^n  - \lambda p_1^n) + \lambda p_2^n.
    \end{aligned}\label{algorithmexplanation2}
  \end{equation}
  Denoting the terms between brackets by $g_l$, one obtains
  \begin{equation}
    \begin{aligned}
      f_1^n &= g_1 +  \lambda p_1^n, \\
      f_2^n &= g_2 +  \lambda p_2^n,
    \end{aligned} \quad \text{ and } \quad
    \begin{aligned}
      g_1 &= -2\lambda p_2^n - g_2, \\
      g_2 &= -2\lambda p_1^n - g_1 .
    \end{aligned} \label{motivation_coupling}
  \end{equation}
The conditions in \eqref{strongiterschemeeqn2}-\eqref{stronggliupdate} are the linearised counterparts of \eqref{motivation_coupling}.
\end{remark}
%
%% REMARK DIFFERENT COUPLING FORMULATIONS
\begin{remark}[different decoupling formulations] \label{RemarkFormulationVariants}~
  The decoupled conditions in \eqref{SemiDiscreteFormulation2}-\eqref{SemiDiscreteFormulation3} can be formulated as convex combinations of the terms $g$ and $p$, namely
  \begin{align*}
    \flux{i} \cdot \vt{n_l} &= (1-\lambda) g^i_l +
    \lambda p_l^{n,i}
%     &&\mbox{ on } \Gamma\times[0,T],
    \tag{\ref{strongiterschemeeqn2}'} \label{strongiterschemeeqn2'} \\[0.9mm]
    (1-\lambda) g_l^i :\hspace{-1mm} &= -2\lambda p_{3-l}^{n,i-1}- (1-\lambda)
    g_{3-l}^{i-1}. && \tag{\ref{stronggliupdate}'} \label{stronggliupdate'}
  \end{align*}
  The convergence analysis below can be carried out for this formulation without any difficulty. 
  However, the DD scheme using this convex formulation showed a slower convergence in the numerical experiments than when  (\ref{strongiterschemeeqn2})-(\ref{stronggliupdate}) was used. 
  Moreover, it is easier to find close to optimal parameters for the latter. 
  Such aspects are discussed in Section \ref{SectionNumericalExperiments}. In view of this, in what follows we restrict the analysis to the initial formulation.
\end{remark}
Before formulating the main result we specify the notation that will be used below.

%The spaces used in the weak formulation of the problem are stated
%
% Definition and Notation of spaces
\begin{notation}%[spaces for weak formulation]
  %We use standard notations in the functional analysis. 
  $L^2(\dom)$ is the space of Lebesgue
  measurable,  square integrable  functions over $\dom$.  $H^1(\dom)$ contains functions in $L^2(\dom)$ having also weak derivatives in $L^2(\dom)$.
  $H_0^1(\dom) = \completion{C_0^\infty(\dom)}{H^1}$, where the completion is with respect to the standard
  $H^1$ norm and $C_0^\infty(\dom)$ is the space of smooth functions with compact support in
  $\dom$. The definition for $H^1(\dom_l)$ ($l = 1, 2$) is similar. With $\Gamma$ being a $(d-1)$ dimensional manifold in $\bar{\dom}$, $H^{\frac{1}{2}}(\Gamma)$ contains the traces of $H^1$ functions on $\Gamma$ (see e.g. \cite{Brezzi1991,MacLean2000,QuarteroniValli2005}. Given $u \in H^1(\dom)$, by its trace on $\Gamma$ is denoted by $\restrictTo{u}{\Gamma}$.
% remooved the citation \cite{Evans2010,QuarteroniValli2005}) because I looked up the stuff on the trace operator and H^{1/2} in Brezzi1991,MacLean2000 and it is precisely contained there. Evans is too basic (albeit the idea being explained nicely. QuarteroniValli2005 explains the stuff, but and Brezzi is exactly our setting and older. At any rate, I looked it up there. 

%   unless it is on the interface in which case we write $\langle \cdot,
%   \cdot   \rangle_\Gamma$.
 Furthermore, the following spaces will be used %\marginpar{why having both $\Fs_l$ and $H_{\oB{l}}^1$? They are the same!: they are not the same, \partial Omega is not the all boundary of Omega_1}
%
%  For $u \in H^1(\dom)$, we will mostly denote its restriction to $\Gamma \subset \partial\dom$ by
%$\restrictTo{u}{\Gamma}$
%instead of $\tr{\Gamma} u$, $\tr{\Gamma} : H^1(\dom) \rightarrow H^{\frac{1}{2}}(\Gamma)$ being the trace operator %cf.
%\cite[Theorem A.2.3 and p. 132 ff]{Berninger2009} or \cite{Brezzi1991,MacLean2000}.
  \begin{align}
    \Fs_l&:= %H_{\oB{l}}^1  :=
    \left\{ u \in H^1(\dom_l)\, \bigl|\bigr. \, \restrictTo{u}{\oB{l}}
    \equiv 0 \right\}, \\%
    % Makro \fs{j}
    %
    \Fs   &:= \left\{ (u_1,u_2) \in \Fs_1\times \Fs_2 \, \bigl|\bigr. \, {u_1}_{|_\Gamma} \equiv {u_2}_{|_\Gamma}
    \right\},\\
    \Tracespace &\hspace{3.1pt}= \bigl\{ \nu \in H^{1/2}(\Gamma) \,\bigl|\bigr.\, \nu = \onGamma{w} \mbox{ for a } w \in
    H_0^1(\dom) \bigr\}.
  \end{align}
  Note, that $\Fs = H^1_0(\dom)$. %by \cite[Lemma 3.2.3]{Berninger2009}.
  $\Tracespace'$ denotes the dual space of $\Tracespace$. $\langle \cdot, \cdot \rangle_X$ will denote the $L^2(X)$ scalar product, with $X$ being one of the sets $\dom$, $\dom_l$ ($l = 1, 2$) or $\Gamma$. Whenever self understood, the notation of the domain of integration $X$ will be dropped. Furthermore, $\spl \cdot , \cdot \spr_\Gamma$ stands also for the duality pairing between $\Tracespace'$ and $\Tracespace$.
\end{notation}

In what follows we make the following
\begin{assumptions} \label{assumptionsondata} With $l=1,2$, we assume that%
%\footnote{similar assumptions are used in the literature, cf. \cite{Pop2004365}} %
%
 \begin{enumerate}[itemsep=0.5ex,label=\alph*)]
  \item $k_l:\RR \rightarrow [0,1]$ are strictly mo\-notonically increasing and
  Lipschitz continuous functions with Lipschitz constants $L_{k_l} > 0$,
%,  such that $k_l(0)=0$.
  \item there exists  $m \in \RR$ such that $ 0 <  m \leq k_1(S)$, $k_2(S)$ for all $S \in \RR$,
  \item $S_l: \RR \rightarrow \RR$ are monotonically increasing and Lipschitz continuous functions with Lipschitz constants
  $L_{S_l}>0$.
 \end{enumerate}%\vspace{-5mm}
\end{assumptions}
For later use we define $L_k:=\max\{$ $L_{k_1}, L_{k_2}\}$ and $L_S:=\max\{L_{S_1}, L_{S_2}\}$.

In a simplified formulation, the main result in this paper is
\begin{theorem} \label{strongmainresult}
  Assume there exists a solution pair $(p_1^n,p_2^n) $ to
  (\ref{SemiDiscreteFormulation1})--(\ref{SemiDiscreteFormulation4}) that additionally fulfils
  $\sup_l\norm{\bm{\nabla} \bigl( p_l^n + z \bigr) }_{L^\infty} \leq M < \infty $. %
  Let $L_l$ obey $L_{S_l} < 2L_{l}$ for $l=1,2$ and assume that the time step
  $\tau > 0$ is chosen small enough, so that for both $l$ one has
  \begin{align}
  \tau <   \frac{2m}{L_{k_l}^2M^2}\left(\frac{1}{L_{S_l}} -\frac{1}{2L_{l}}\right).
  \end{align}
 Then the sequence of solution pairs $\bigl\{(p_1^{n,i}, p_2^{n,i})\bigr\}_{i \ge 1}$ of
  (\ref{strongiterschemeeqn1})--(\ref{strongiterschemeeqn2}) converges to $(p_1^n,p_2^n) $.
  \end{theorem}
\begin{remark}
  The precise form of Theorem \ref{strongmainresult} will be formulated in Section \ref{SectionAnalysisOfScheme}, after having defined a weak solution. The theorem above is given for the ease of presentation.
  %, and in case the reader is less interested technical details.
\end{remark}

\section{Analysis of the scheme.} \label{SectionAnalysisOfScheme}

This section gives the convergence proof for the proposed scheme. %This is based on the weak formulation of the semi-dis\-crete problems, and equivalent forms of it.
%
%\subsubsection*{Weak Problem Statement}
%% weak formulation conformal case
The starting point is the Euler implicit discretisation in Section \ref{SectionProblemFormulationAndIterationScheme}. Assuming
$\bigl(p_1^{n-1},p_2^{n-1}\bigr) \in \Fs$ to be known, a weak formulation of (\ref{SemiDiscreteFormulation1})--(\ref{SemiDiscreteFormulation4})
is given by
% Semi-Discretized Problem weak form
\begin{problem}[Semi-discrete weak formulation]~ \label{WeakSemiDiscreteFormulation}
  Find $(p_1^n$, $p_2^n)\in \Fs$ such that $\Fln{l}\cdot
  \vt{n_l} \in \Tracespace' $ for $l=1,2$ and
  \begin{align}
    \spl S_l(p_l^n),\varphi_l \spr - \tau\spl \Fln{l},	\bm{\nabla} \varphi_l\spr
    + \tau\spl \Fln{3-l}&\cdot
  \vt{n_l},\onGamma{\varphi_l} \spr_\Gamma %\nonumber\\%
    = \spl S_1(p_1^{n-1}),\varphi_1 \spr, \label{weakformulationequation1}
  \end{align}
  for all $(\varphi_1,\varphi_2) \in \Fs$.
\end{problem}

\begin{remark}
If $(p_1^n,p_2^n)\in \Fs$ is a solution of Problem \ref{WeakSemiDiscreteFormulation},  we have
    $\onGamma{p_1^n} =
    \onGamma{p_2^n}$ by definition of $\Fs$.
    Testing in (\ref{weakformulationequation1}) by an arbitrary $\varphi_l \in C_0^\infty(\dom_l)$ shows that the distribution $\dv\Fln{l}$ is regular and in $L^2$, yielding $ \Fln{l} \in H(\mbox{div},\dom_l)$ and
  \begin{align}
    S_l(p_l^n) - S_l(p_l^{n-1}) &= -\tau\dv \Fln{l}
  \qquad \mbox{a. e. in } \dom_l
    \label{fluxIsInHdiv}
  \end{align}
  by the variational lemma.  By Lemma III. 1.1 in \cite{Brezzi1991}, $\Fln{l} \cdot \vt{n_l}
  \in H^{-1/2}(\partial\dom_l)$ and integrating by parts in
  (\ref{weakformulationequation1}) yields
  \begin{align}
    0 &=  -\spl \Fln{l} \cdot \vt{n_l},\onGamma{\varphi_l} \spr_\Gamma %
	+ \spl \Fln{3-l} \cdot \vt{n_l},\onGamma{\varphi_l} \spr_\Gamma \label{weakFluxContinuity}
  \end{align}
  for all $(\varphi_1,\varphi_2) \in \Fs$.
  Therefore
  \begin{align}
    \Fln{l} \cdot \vt{n_l}
	= \Fln{3-l} \cdot \vt{n_l}
	\label{weakFluxContinuityTracespace}
  \end{align}
   in $\Tracespace'$ since the trace is a surjective operator. 
% Finally, (\ref{weakFluxContinuityBerninger}) follows from \eqref{weakFluxContinuityTracespace} by testing in
% \eqref{weakformulationequation1} with $(R_1\nu,R_2\nu)$ for a $\nu \in \Tracespace$.
  \end{remark}
   Note additionally that Problem \ref{WeakSemiDiscreteFormulation} is equivalent to the semi-discrete Richards equation on the whole domain, namely to find $(p_1^n,p_2^n)$ $\in \Fs$ such that
  \begin{align}
    \spl S_1(p_1^n)&,\varphi_1 \spr - \tau\spl \Fln{1},\bm{\nabla} \varphi_1\spr  %
    + \spl S_2(p_2^n),\varphi_2 \spr - \tau \spl \Fln{2},\bm{\nabla} \varphi_2 \spr
    %\nonumber\\
  %
    %
%     &
    = \spl S_1(p_1^{n-1}),\varphi_1 \spr + \spl S_2(p_2^{n-1}),\varphi_2 \spr,
    \label{WeakSemiDiscreteFormulationSummed}
  \end{align}
  for all $(\varphi_1,\varphi_2) \in \Fs$. %C.f. \cite[Proposition 2.5]{Berninger2014}.
%\end{remark}
%

\begin{remark} By applying a Kirchhoff transform in each subdomain $\dom_l$, Problem \ref{WeakSemiDiscreteFormulation} can be reformulated as a nonlinear transmission problem. The existence and uniqueness of a solution for such problems has been studied in \cite{Jaeger1998,Jaeger2002} for the case when $\dom_1$ is surrounded by $\dom_2$, and the common boundary is smooth, however.
\end{remark}

Now we can give the weak form of the iterative scheme. %To decouple the problems in the two subdomains, we follow the ideas in \cite{Lions1988} and replace the flux continuity condition (\ref{weakformulationequation1}) by two Robin type ones on each side of $\Gamma$. The iterative scheme is defined by linearising the flux and introducing the stabilising $L$-terms. This is a linearised counterpart of the schemes based on nonlinear Robin-type conditions as discussed in \cite{Berninger2014}.
Let  $n \in \mathbb N$ and assume that the pair $\bigl(p_1^{n-1}$, $p_2^{n-1}\bigr) \in \Fs$ is given. Furthermore, let $\lambda > 0$ and $L_l > 0$ ($l = 1, 2$) be fixed parameters and
  $$
  \plni{0}:= p_l^{n-1}, \quad \text{ as well as } \quad
  \gli{0}:=\Fln[n-1]{l}\cdot\vt{n_l} -
  \lambda \restrictTo{p_l^{n-1}}{\Gamma}.$$
The iterative scheme is defined through
%% ITERATION SCHEME WEAK FORM
\begin{problem}[$L$-scheme, weak form] \label{WeakIterScheme}
%  Let the time discretisation be as in Section \ref{SectionProblemFormulationAndIterationScheme},
 Let $i \in \NN$ and assume that the approximations $\bigl\{\plni{k}\bigr\}_{k=0}^{i-1}$ and
  $\bigl\{\gli{k}\bigr\}_{k=0}^{i-1}$ are known for $l=1,2$. Find $\bigl(\plni[1]{i},\plni[2]{i}\bigr) \in \Fs$
  such that 
  \begin{align}
    L_l\spl p_l^{n,i},\varphi_l \spr
    - \tau\spl \flux{i},\bm{\nabla}\varphi_l \spr
    + \tau \spl \lambda \plni{i} + \gli{i},\varphi_l
\spr_\Gamma  %\nonumber \\
    %
%     &
    &= L_l\spl p_l^{n,i-1},\varphi_l\spr - \spl S_l(p_l^{n,i-1}) -
    S_l\bigl(p_l^{n-1}\bigr),\varphi_l\spr  %
    \label{weakiterschemeeqn1}  \\
    \spl g_l^i,\varphi_l \spr_\Gamma\,\,
    &\hspace{-3pt} :=\spl -2\lambda
    {p_{3-l}^{n,i-1}}-  g_{3-l}^{i-1},\varphi_l \spr_\Gamma   && \label{weakgliupdate}
  \end{align}
  %    \end{equation*}
  holds for all $(\varphi_1,\varphi_2)\in \Fs$.
\end{problem}
%
%\begin{remark}
%Observe that the pressure term in (\ref{weakgliupdate}) is interpreted as a functional on $\Tracespace$ via the
%$L^2$-scalar product of $L^2(\Gamma)$.
%\end{remark}

\subsection{Intuitive justification of the  \texorpdfstring{$L$}{L}-scheme}
We start the analysis by taking a closer look at the formal limit of the $L$-scheme iterations in weak form and show that this is actually a
reformulation of Problem \ref{WeakSemiDiscreteFormulation}. %Therefore, if the $L$-scheme converges to this limit system it solves Problem \ref{WeakSemiDiscreteFormulation}.

\begin{lemma}[Limit of the \texorpdfstring{$L$}{L}-scheme]\label{iterationinterpretationlemma}
  Let $n \in \NN$ be fixed and assume that the functions $p_l^n\in \Fs_l$ and $g_l\in \Tracespace'$ ($l = 1, 2$) exist such that
  %$p_1^n,p_2^n, g_1,g_2$ are a solution of the system
    \begin{align}
      \spl S_l(p_l^n),\varphi_l\spr - \spl S_l\bigl(p_l^{n-1}\bigr)&,\varphi_l\spr
      - \tau\spl \Fln{l},\bm{\nabla}\varphi_l \spr %\nonumber\\
      %&
      + \tau \spl \lambda p^n_l,+ g_l, \varphi_l \spr_\Gamma  %\nonumber
      = 0, 	%
      \label{weakiterschemelimiteqn1} 	\\
      \spl g_l,\varphi_l \spr_\Gamma  =\spl &-2\lambda p_{3-l}^{n} -
      g_{3-l}, \varphi_l \spr_\Gamma, 	&& \label{weakiterschemelimiteqn2}
      \end{align}
      hold for all $(\varphi_1,\varphi_2) \in \Fs$.  Then the interface conditions
      \begin{align}
	\onGamma{ p_1^n } 		&= \onGamma{ p_2^n }	&& \mbox{ in } \Tracespace, \\
	\Fln{1} \cdot \vt{n_1}%
	  &= \Fln{2} \cdot \vt{n_1} && \mbox{ in } \Tracespace'		
\label{FluxContinuity}
      \end{align}
      are satisfied and $(p_1^n,p_2^n)$ solves Problem \ref{WeakSemiDiscreteFormulation}.
      Moreover,
      \begin{align}
	g_l = -\lambda \onGamma{ p_l^n } + \Fln{l} \cdot
\vt{n_l}  \label{Gequation}
      \end{align}
      in $\Tracespace'$.
      Conversely, if $(p_1^n,p_2^n)\in \Fs$ is a solution of Problem \ref{WeakSemiDiscreteFormulation} and 
      $ g_l := -\lambda \onGamma{ p_l^n } + \Fln{l}
      \cdot \vt{n_l}$, then $p_l^n$ and $g_l$ solve the system
      (\ref{weakiterschemelimiteqn1}),  (\ref{weakiterschemelimiteqn2}).
\end{lemma}
\begin{remark}
 Lemma \ref{iterationinterpretationlemma} states that solving Problem \ref{WeakSemiDiscreteFormulation}
is equivalent to finding a solution to (\ref{weakiterschemelimiteqn1}), (\ref{weakiterschemelimiteqn2}). This reformulation will be used to show, that the L-scheme converges to a solution of Problem \ref{WeakSemiDiscreteFormulation}
% , namely that as in Remark \ref{rem:idea}, the limit will be a solution of (\ref{weakiterschemelimiteqn1}), (\ref{weakiterschemelimiteqn2}) and consequently of Problem \ref{WeakSemiDiscreteFormulation}.
\end{remark}
%
% proof of previous Lemma
\begin{proof}%[of \cref{iterationinterpretationlemma}]
  Writing out (\ref{weakiterschemelimiteqn2}) for $l=1,2$ and subtracting the resulting equations 
  yields $\onGamma{ p_1^n } = \onGamma{ p_2^n}$ in the sense of traces.
  On the other hand, adding up these equations leads to %
  $ (g_1 + g_2) = -\lambda (\onGamma{ p_1^n } + \onGamma{ p_2^n })$. Inserting this into the sum of the equations
  (\ref{weakiterschemelimiteqn1}) leads to (\ref{WeakSemiDiscreteFormulationSummed}), and by equivalence to the
  semi-discrete formulation (\ref{weakformulationequation1}). Moreover, by (\ref{fluxIsInHdiv}) one has  $S_l(p_l^n) - S_l(p_l^{n-1}) = -\tau\dv \Fln{l}
  $ a.e. and therefore integrating by parts in (\ref{weakiterschemelimiteqn1}) gives
  %\begin{align*}
    $ g_l = -\lambda \onGamma{ p_l^n } + \Fln{l} \cdot \vt{n_l}$ in $\Tracespace'$.
    %= -\lambda p_l^n - k_{3-l}\bigl(S_{3-l}(p_{3-l}^n)\bigr)\partial_{\vt{n_l}} p_{3-l}^n
  %  \end{align*}

  Conversely, if $(p_1^n,p_2^n)$ solves Problem \ref{WeakSemiDiscreteFormulation}, then
  $\onGamma{ p_1^n } = \onGamma{ p_2^n }$ and  
  \begin{align}
      g_l = - \lambda \onGamma{ p_l^n } &+ \Fln{l} \cdot \vt{ n_l } %
 %     = - \lambda \onGamma{ p_{3-l}^n } + \Fln{l} \cdot \vt{n_l} \nonumber \\
    %
 %   &= -2 \lambda \onGamma{ p_{3-l}^n } -\bigl(- \lambda \onGamma{ p_{3-l}^n} %
 %     + \Fln{3-l} \cdot \vt{ n_{3-l} } \bigr)  \nonumber \\
    %
= -\lambda \onGamma{ p_{3-l}^n } + \Fln{3-l} \cdot \vt{ n_{3-l} }
%    &= -2 \lambda \onGamma{ p_{3-l}^n } -   g_{3-l} \nonumber
= -2 \lambda \onGamma{ p_{3-l}^n } -   g_{3-l}
  \end{align}
%  where we used the flux continuity assured by (\ref{weakFluxContinuityTracespace}).
  is deduced by the flux continuity (\ref{weakFluxContinuityTracespace}).
 Finally, (\ref{weakiterschemelimiteqn1}) now follows by integrating (\ref{fluxIsInHdiv}) by parts and using
  the definition of  $g_l$.
\end{proof}

%
%% Convergence of iterative scheme
\subsection{Convergence of the scheme} \label{SubsectionConvergenceOfScheme}

The convergence of the $L$-scheme involves two steps: first, we prove the existence and uniqueness of a solution to
Problem \ref{WeakIterScheme} defining the linear iterations, and then we prove the convergence of the sequence of such solutions to the expected limit.
%
%\cref{WeakSemiDiscreteFormulation} thus solving our problem.
%
% EXISTENCE OF SOLUTIONS TO THE ITERATION SCHEME
\begin{lemma} \label{WeakIterSchemeStepExistenceLemma}
 Problem \ref{WeakIterScheme} has a unique solution.
\end{lemma}
\begin{proof}
  This is a direct consequence of the Lax-Milgram lemma. 
%  
%  We use the Hilbert space ${\Hr}_l = \Fs_1 \times \Fs_2$ equipped with the scalar product
%     $\tspl u,v \tspr_{{\Hr}_l} = \tspl u,v \tspr_{H^1} + \frac{1}{2}\tspl u,v \tspr_\Gamma$, and the corresponding norm.
%   Observe that Problem \ref{WeakIterScheme} can be written as
%   \begin{align}
%     a_l(u_l,v_l)= F_l(v_l), \label{laxmilgramformulation}
%   \end{align}
% where $a_l : {\Hr}_l \times {\Hr}_l \rightarrow \RR$ is a bounded and coercive bilinear form defined by
%   \begin{align}
%     a_l(u,v) := \tau\spl k_l\bigl(S_l&(p_l^{n,i-1})\bigr)\vt{\bm{\nabla} u},\bm{\nabla} v \spr
% %\nonumber \\
% %	      &
% + L_l\spl u,v \spr
% 	      + \tau \lambda \spl \onGamma{ u }, \onGamma{ v } \spr_\Gamma, \label{a_l}
%   \end{align}
%  and $F_l : {\Hr}_l \rightarrow \RR$ is the bounded linear functional
%   \begin{align*}
%     F_l(v) := L_l\spl p_l^{n,i-1}&,v\spr - \spl S_l(p_l^{n,i-1}),v\spr + \spl S_l\bigl(p_l^{n-1}\bigr),v\spr
%  %   \nonumber \\
%     %
%    % &
%     - \tau  \spl\gli{i},v\spr_\Gamma
%     - \tau\spl k_l\bigl(S_l(p_l^{n,i-1})\bigr)\vt{\nabla z},\bm{\nabla} v \spr .
%   \end{align*}
\end{proof}

We now prove the convergence result, which was announced in Theorem \ref{strongmainresult}. We assume that the solution $\bigl(p_1^{n-1},p_2^{n-1}\bigr)$ of Problem \ref{WeakSemiDiscreteFormulation} at time step $(n-1)$ is known and let $p_l^{n,0} \in \Fs_l$ be arbitrary starting pressures (however, a natural choice is $p_l^{n,0} := p_l^{n-1}$).

Lemma \ref{WeakIterSchemeStepExistenceLemma} enables us to construct a sequence  $\bigl\{\plni{i}\bigr\}_{i \in \NN_0}
\in \Fs_l^\NN$ of solutions to Problem \ref{WeakIterScheme}  and prove its convergence to the solution
$\bigl(p_1^{n},p_2^{n}\bigr)$ of Problem \ref{WeakSemiDiscreteFormulation} at the subsequent time step.
%Recall Assumption \ref{assumptionsondata} for all the constants appearing below.

\begin{theorem}[Convergence of the DD scheme] \label{mainresult}
  Assume there exists a solution $(p_1^n,p_2^n)\in \Fs$ to Problem \ref{WeakSemiDiscreteFormulation} s.t. $\sup_l\norm{\bm{\nabla} \bigl( p_l^n + z \bigr) }_{L^\infty} \leq M < \infty $ and let $g_l$ be as in (\ref{Gequation}).  %
  Let Assumptions \ref{assumptionsondata} hold, $\lambda > 0$ and $L_l\in \RR $ be given with $L_{S_l}/2 < L_{l}$ for $l=1,2$. For arbitrary starting pressures $p_l^{n,0} := v_{l,0}\in \Fs_l$ $(l=1,2)$ let $\bigl\{(p_1^{n, i}, p_2^{n, i})\bigr\}_{i \in \NN_0}$  be the sequence of solutions of Problem \ref{WeakIterScheme} and let $\bigl\{\gli{i}\bigr\}_{i \in \NN_0}$ be defined by (\ref{weakgliupdate}).
  Assume further that the time step $\tau$ satisfies
  \begin{align}
    \tau < \frac{2m}{L_{k_l}^2M^2}\left(\frac{1}{L_{S_l}} -\frac{1}{2L_{l}}\right).   \label{timestepRestriction}
  \end{align}
  Then $\plni{i} \rightarrow p_l^n$ in $\Fs_l$ and $\gli{i} \rightarrow g_l$ in $\Fs_l'$ as $i\rightarrow \infty$ for
$l=1,2$.  %In other words, the iterative scheme, Problem \ref{WeakIterScheme}, converges to a solution of Problem \ref{WeakSemiDiscreteFormulation}.
\end{theorem}

 \begin{remark} The essential boundedness of the pressure gradients can be proven under the additional assumption that the functions $S_l$ are strictly increasing and the domain is of class $C^{1,\alpha}$, see e.g. \cite[Lemma 2.1]{Cao201525}.   %
 \end{remark}

\begin{proof}
  We introduce the iteration errors $\epli{i}	:= p_l^n - \plni{i}$ as well as $\egli{i}	:= g_l^n - \gli{i}$,
  add
  $L_l\tspl p_l^n,\varphi_l\tspr - L_l\tspl p_l^n,\varphi_l\tspr$ to (\ref{weakiterschemelimiteqn1}) and subtract
  (\ref{weakiterschemeeqn1}) to arrive at
  \begin{align}
    L_l&\spl \epli{i},\varphi_l \spr %
      + \tau  \lambda \spl \epli{i},  \varphi_l  \spr_\Gamma%
      + \tau  \spl \egli{i}, \varphi_l  \spr_\Gamma  %\nonumber \\%
    %&
    + \tau \Bigl[ \bspl -\Fln{l}
    {%\color{\addnullcolor}
    -\, \klni{i-1} \bm{\nabla} \bigl( p_l^n + z \bigr) } \nonumber \\
    &{%\color{\addnullcolor}
    \,+\, \klni{i-1}\bm{\nabla} \bigl( p_l^n + z \bigr) }
    + \flux{i},\bm{\nabla}\varphi_l\bspr  \Bigr]		%\nonumber \\
    %&
    = L_l\spl \epli{i-1},\varphi_l \spr
    - \spl S_l(p_l^n)- S_l(\plni{i-1}), \varphi_l \spr.
   \label{mainproofstep1}
%     \underbrace{- \spl S_l(p_l^n) -S_l(p_l^{n-1}), \varphi_l \spr %
%     + \spl S_l(\plni{i-1}) -S_l(p_l^{n-1}), \varphi_l \spr}_{%
%     - \spl S_l(p_l^n)- S_l(\plni{i-1}), \varphi_l \spr}.
  \end{align}
  Inserting $\varphi_l := \epli{i}$ in (\ref{mainproofstep1}) and noting that
  \begin{align}
  L_l\bspl\epli{i}-\epli{i-1},\epli{i}\bspr
    &= \frac{L_l}{2}\Bigl[ \bigl\|\epli{i}\bigr\|^2 %\nonumber \\
    %&\hspace{0.7cm}
    - \bigl\|\epli{i-1}\bigr\|^2
    + \bigl\|\epli{i}-\epli{i-1}\bigr\|^2\Bigr], \nonumber %\label{mainproofstep1intermediate}
  \end{align}
  yields
%   \ifoptionaltext
%   {\optionaltextcolor
%   \begin{align}
%     L_l\bigl\|\epli{i}& \bigl\|^2 %
%       + \tau \lambda \bigl\| \epli{i}\bigr\|_\Gamma^2 \nonumber\\%%
%     &+ \tau \bspl\Bigl(\kln{l} -\klni{i-1}\!\Bigr)\vt{ \bm{\nabla} \bigl( p_l^{n-1} - z \bigr) },\gradepli{i}\bspr
%       +\tau \bspl\klni{i-1}\gradepli{i},\gradepli{i}\bspr 	\nonumber \\
%     &= L_l\spl \epli{i-1},\epli{i} \spr 	
% 	- \spl S_l(p_l^n)- S_l(\plni{i-1}), \epli{i} \spr
% 	- \tau  \spl \egli{i},\epli{i} \spr_\Gamma				. \label{mainproofstep1extra}
%   \end{align}
%   as intermediate step and thus with (\ref{mainproofstep1intermediate})%
%   }
%   \fi
  \begin{align}
    \frac{L_l}{2}&\Bigl[ \bigl\|\epli{i}\bigr\|^2 %
      - \bigl\|\epli{i-1}\bigr\|^2
      + \bigl\|\epli{i}-\epli{i-1}\bigr\|^2\Bigr]
  %  \nonumber \\
      %
   % &
    \underbrace{{%\color{\addnullcolor}
    +\spl S_l(p_l^n)
      -S_l(\plni{i-1}), \epli{i-1}\spr}}_{=: I_1} %
      + \tau \lambda \spl \epliOnGamma{i}, \epliOnGamma{i} \spr_\Gamma %
    \nonumber\\%
    &= \underbrace{\spl S_l(p_l^n)- S_l(\plni{i-1}),
      {%\color{\addnullcolor}
      \epli{i-1}}-\epli{i}
      \spr}_{=: I_2}  - \tau     \spl \egli{i},\epliOnGamma{i} \spr_\Gamma %
    \nonumber \\	
    &-\underbrace{\tau \bspl\Bigl(\kln{l} -\klni{i-1}\!%
      \Bigr)\bm{\nabla} \bigl( p_l^n + z \bigr),\gradepli{i}\bspr}_{=:   I_3}
%     \nonumber \\
    %
%     &
    -\underbrace{\tau \bspl\klni{i-1}\gradepli{i}
    ,\gradepli{i}\bspr.}_{=:I_4}
    \label{mainproofstep2}
  \end{align}
 We estimate now the terms $I_1$--$I_4$ in (\ref{mainproofstep2}) one by one.   By Assumption \ref{assumptionsondata}{c)}, for $I_1$ we have
%%% ausfuehrliches Argument, lass ich mal drin
%   We assumed that there exists a $L_{S_l} > 0$ such that
%$\bigl|S_l(p_l^n) -
%   \Splni{i-1}\bigr| < L_{S_l}\bigl|\epli{i}\bigr|$
%
%   as well as monotonicity  for the saturation such that we
%have %
%   $\bigl( S_l(p_l^n) - \Splni{i-1}\bigr) \epli{i-1} \geq
%0$. %
%   Then,
  \begin{align}
  \frac{1}{L_{S_l}} \bigl\| S_l(p_l^n) - \Splni{i-1}\bigr\|^2
    &\leq \spl S_l(p_l^n) - \Splni{i-1}, \epli{i-1}\spr.\label{I4estimateitem}
  \end{align}
%
%  \begin{enumerate}[itemsep=-0.5ex,label=(\roman*)]
    %\item
 $I_2$ is estimated by 
  \begin{align}
    \bigl|I_2\bigr|
      &= \left|\bspl S_l(p_l^n)- S_l(\plni{i-1}), \epli{i-1}-\epli{i} \bspr \right| %\nonumber \\%
%
%       &
      \leq \frac{L_{l}}{2}\bigl\|\epli{i-1}-\epli{i} \bigr\|^2%
      + \frac{1}{2L_{l}}  \bigl\| S_l(p_l^n)%
      - S_l(\plni{i-1})\bigr\|^2 .\label{I1estimate}
  \end{align}
%
%  \item
  For an $\epsilon_l > 0$ to be chosen below we use Young's inequality to deal with $I_3$, which can be estimated by
  \begin{align}
    \bigl|I_3\bigr| &= \Bigl|\tau \bspl\Bigl(\kln{l}%
    -\klni{i-1}\!\Bigr)\vt{ \bm{\nabla}} \bigl( p_l^n
    + z \bigr) ,\gradepli{i}\bspr\Bigr| \nonumber \\%
    &\leq \tau \bigl\|\bigl(\kln{l} -\klni{i-1}\bigr)\bm{\nabla} \bigl( p_l^n + z \bigr) \bigr\|
  \bigl\|\gradepli{i}\bigr\| 		\nonumber \\
    &
    \leq \tau L_{k_l}M \bigl\|S_l(p_l^n)-S_l(\plni{i-1})\bigr\| \bigl\|\gradepli{i}\bigr\| \nonumber \\%
    &\leq \tau L_{k_l}M\epsilon_l\bigl\|S_l(p_l^n)-S_l(\plni{i-1})\bigr\|^2 + \tau\frac{L_{k_l}M}{4\epsilon_l}
  \bigl\|\gradepli{i}\bigr\|^2, \label{mainproofstep2b}
  \end{align}
  where we used the Lipschitz continuity of $k_l$ and the assumption $\sup_l\bigl\| \bm{\nabla} \bigl( p_l^n + z \bigr)  \bigr\|_{L^\infty}< M$.
  %
%    \item\label{saturationislipschitz}
Finally, by Assumption \ref{assumptionsondata}{b)} one has
  \begin{align}
    \tau \bspl\klni{i-1}\gradepli{i},\gradepli{i}\bspr \geq \tau m \bigl\|\gradepli{i}  \bigr\|^2
    \label{I4estimate}
  \end{align}
  for $I_4$.
%  \end{enumerate}
  Using the estimates (\ref{I4estimateitem})--(\ref{I4estimate}), (\ref{mainproofstep2}) becomes
  \ifoptionaltext
    {\optionaltextcolor
      \begin{align}
	\frac{L_l}{2}\Bigl[ &\bigl\|\epli{i}\bigr\|^2 %
	  - \bigl\|\epli{i-1}\bigr\|^2
	  + \bigl\|\epli{i}-\epli{i-1}\bigr\|^2\Bigr]
	  +\spl S_l(p_l^n)- S_l(\plni{i-1}), \epli{i-1}\spr %
	  \nonumber \\
	  +& \tau \lambda \spl \epliOnGamma{i}, \epliOnGamma{i} \spr_\Gamma \nonumber\\%%
	&= \spl S_l(p_l^n)- S_l(\plni{i-1}), \epli{i-1}-\epli{i} \spr
	- \tau \spl \egli{i},\epliOnGamma{i} \spr_\Gamma			 	\nonumber \\	
	&-\tau \bspl\Bigl(\kln{l} -\klni{i-1}\!\Bigr)\vt{ \bm{\nabla} \bigl( p_l^{n-1} - z \bigr) },\gradepli{i}\bspr
	  -\tau \bspl\klni{i-1}\gradepli{i},\gradepli{i}\bspr.
	%\label{mainproofstep2}
      \end{align}
    }
  \fi
  \begin{align}
    \frac{L_l}{2}\Bigl[& \bigl\|\epli{i}\bigr\|^2 %
      - \bigl\|\epli{i-1}\bigr\|^2
      \Bigr]
      +\frac{1}{L_{S_l}} \bigl\| S_l(p_l^n) - \Splni{i-1}\bigr\|^2%
      %\nonumber \\
      %
      %&
      + \tau \lambda \spl \epliOnGamma{i}, \epliOnGamma{i} \spr_\Gamma %
      + \tau \spl \egli{i},\epliOnGamma{i} \spr_\Gamma			\nonumber\\%%
    &\leq \left( \frac{1}{2L_{l}} +\tau L_{k_l}M\epsilon_l \right)
    \bigl\|S_l(p_l^n)-S_l(\plni{i-1})\bigr\|^2 	
%\nonumber \\	
    %
    %&
    + \tau \left(\frac{L_{k_l}M}{4\epsilon_l}
      - m \right)\bigl\|\gradepli{i} \bigr\|^2.
    \tag{\ref{mainproofstep2}'}\label{mainproofstep3}
  \end{align}
  In order to deal with the interface term $\tau \spl \egli{i},\epliOnGamma{i} \spr_\Gamma$ recall, that $\spl \cdot, \cdot \spr_\Gamma$ denotes the dual pairing of $\Tracespace'$ and $\Tracespace$ and the $\Tracespace$-norm simultaneously. 
%   we first note that
%   $\Tracespace$ is a Hilbert space (see \cite[Prop. 2.3]{Berninger2014}) % and \cite[Prop. 3.2.1]{Berninger2009})%
  %
%  \footnote{see also \cite{MacLean2000} for more detailed definitions.}
  %
%    and therefore, by the Fr{\' e}chet-Riesz theorem, for every functional $ F_g \in \Tracespace'$ there exists a $ g \in \Tracespace$ such that $F_g(v) = \spl g, v
%   \spr_{\Tracespace}$ for all $v \in \Tracespace$.  
Subtracting (\ref{weakgliupdate}) from (\ref{weakiterschemelimiteqn2}), %
  i.e. $ \egli{i} = -2\lambda \epli[3-l]{i-1} -  \egli[3-l]{i-1}$,
%   , and denoting $ \spl g, v \spr_\tGamma := \spl g,
% v   \spr_{\Tracespace}$, ($\norm{\cdot}_\tGamma$ corresponding\-ly) as well as $F_g = g$ by abuse of notation,%
%   %
%   \footnote{we then have $\spl F_g, v \spr_\Gamma = F_g (v) = \spl g, v \spr_\tGamma = \spl F_g, v \spr_\tGamma$, where
%   the last equality incorporates the abuse of notation.}
  %
  we  get
  \begin{align}
    \ifoptionaltext
    { \optionaltextcolor
    \spl \egli{i+1}, \egli{i+1} \spr_\Gamma }&{\optionaltextcolor%
    = 4\lambda^2 \spl \epli[3-l]{i}, \epli[3-l]{i} \spr_\Gamma %
    + \spl \egli[3-l]{i}, \egli[3-l]{i}
    \spr_\Gamma %
    + 4\lambda  \spl \epli[3-l]{i}, \egli[3-l]{i}\spr_\Gamma}, \mbox{
  \optionaltextcolor thus }\nonumber \\
    \fi
    \bigl\| \epli[l]{i} \bigr\|^2_\tGamma  &= \frac{1}{  4\lambda^2 } %
    \left(\bigl\| \egli[3-l]{i+1} \bigr\|^2_\tGamma %
    -\bigl\| \egli[l]{i} \bigr\|^2_\tGamma %
    - 4\lambda  \spl \epli[l]{i}, \egli[l]{i}\spr_\tGamma \right).
    \label{boundarytrickpre}
  \end{align}
  With $b\in\{p,g\}$ we let $e_b^i := (e_{b,1}^i,e_{b,2}^i) \in \Fs_1 \times \Fs_2$ and
  $\norm{e_b^i}^2:= \sum_{l=1} \norm{e_{b,l}^i}^2$. Similarly, on $\Gamma$ we let $\spl e_{b}^i,e_{b}^i \spr_\tGamma :=  \sum_{l=1}^2 \spl
    e_{b,l}^i,e_{b,l}^i \spr_\tGamma$ and correspondingly
  $\norm{e_b^i}_\tGamma^2%:= \norm{e_b^i}^2_{\Tracespace \oplus \Tracespace }
    = \sum_{l=1}^2 \norm{e_{b,l}^i}_\tGamma^2$. Summing in (\ref{boundarytrickpre}) over $l = 1, 2$ gets
  \begin{align}
    \norm{e_p^i}^2_\tGamma  &= \frac{1}{  4\lambda^2 } %
    \Bigl(\norm{e_g^{i+1}}_\tGamma^2  %
      -\norm{e_g^{i}}^2_\tGamma  %
      - 4\lambda  \spl e_p^{i}, e_g^{i}\spr_\tGamma \Bigr).
      \label{boundarytrick}
  \end{align}
  Doing the same for (\ref{mainproofstep3}) and inserting (\ref{boundarytrick}), leaves us with
  \begin{align}
    \frac{L_l}{2}&\Bigl[ \bigl\|e_p^i\bigr\|^2 %
      - \bigl\|e_p^{i-1}\bigr\|^2
      \Bigr]
      +\sum_{l=1}^2\frac{1}{L_{S_l}} \bigl\| S_l(p_l^n) - \Splni{i-1}\bigr\|^2%
      \nonumber \\
      &+  \frac{\tau}{  4\lambda} %
	\left(\norm{e_g^{i+1}}_\tGamma^2  %
	  -\norm{e_g^{i}}^2_\tGamma %
	  \right) %\nonumber \\
      + \tau\sum_{l=1}^2\biggl( m -\frac{L_{k_l}M}{4\epsilon_l} \biggr) \bigl\|\gradepli{i}\bigr\|^2 \nonumber\\%
    &\leq \sum_{l=1}^2\left(\frac{1}{2L_{l}} +\tau L_{k_l}M\epsilon_l\right)  \bigl\| S_l(p_l^n)-
    S_l(\plni{i-1})\bigr\|^2.
    \label{mainproofstep4}
  \end{align}
  \ifoptionaltext
  {\optionaltextcolor
  or equivalently
  \begin{align}
    &\sum_{l=1}^2\left(\frac{1}{L_{S_l}} -\frac{1}{2L_{l}} -\tau L_{k_l}M\epsilon_l \right) \bigl\| S_l(p_l^n)
  - \Splni{i-1}\bigr\|^2%
      + \tau \sum_{l=1}^2\left(m -\frac{L_{k_l}M}{4\epsilon_l}\right)\bigl\|\gradepli{i} \bigr\|^2
  \nonumber\\%%
    &\leq - \frac{L_l}{2}\Bigl[ \bigl\|e_p^i\bigr\|^2 %
    - \bigl\|e_p^{i-1}\bigr\|^2
    \Bigr]
    -  \frac{\tau}{  4\lambda} %
    \left(\norm{e_g^{i+1}}_\Gamma^2  %
      -\norm{e_g^{i}}^2_\Gamma %
    \right)%
  \end{align}
  }
  \fi
  Now, summing for the iteration index $i=1,\dots,r$ and noticing telescopic sums one gets
  \begin{align}
    \sum_{i=1}^r&\sum_{l=1}^2\Bigl(\frac{1}{L_{S_l}} -\frac{1}{2L_{l}} -\tau L_{k_l}M\epsilon_l \Bigr)
    \bigl\| S_l(p_l^n) - \Splni{i-1}\bigr\|^2%
    \nonumber \\
    &
    + \tau \sum_{i=1}^r\sum_{l=1}^2\left(m -\frac{L_{k_l}M}{4\epsilon_l}\right)\bigl\|\gradepli{i} \bigr\|^2 	
    \nonumber\\%%
    &\leq  \frac{L_l}{2}\Bigl[ \bigl\|e_p^0\bigr\|^2 %
    - \bigl\|e_p^{r}\bigr\|^2
      \Bigr]
	+  \frac{\tau}{  4\lambda} %
    \Bigl(\norm{e_g^{1}}_\tGamma^2  %
      -\norm{e_g^{r+1}}^2_\tGamma %
    \Bigr).%
    \label{mainproofstep5}
  \end{align}
  Now we choose $\epsilon_l = \frac{L_{k_l}M}{2m}$, hence $m -\frac{L_{k_l}M}{4\epsilon_l} = \frac{m}{2} > 0$ for both $l$.
  Recalling the restriction on $L_l$, $\frac{1}{L_{S_l}}-\frac{1}{2L_{l}}
  > 0$, as well as that by the time step restriction $\frac{1}{L_{S_l}} -\frac{1}{2L_{l}} -\tau \frac{L_{k_l}^2M^2}{2m} > 0$ for $l=1,2$, the estimates 
  \begin{align}
    %\lim_{r\rightarrow\infty}
    \sum_{i=1}^r\sum_{l=1}^2\biggl(\frac{1}{L_{S_l}} -\frac{1}{2L_{l}} &-\tau
    \frac{L_{k_l}^2M^2}{2m} \biggr) \bigl\| S_l(p_l^n) - \Splni{i-1}\bigr\|^2
    %\nonumber \\
    %&
    \leq  \frac{L_l}{2}\bigl\|e_p^0\bigr\|^2 %
    +  \frac{\tau}{  4\lambda}\norm{e_g^{1}}_\Gamma^2 , \label{mainproofstep6a} \\%
    %
    %+
    %\lim_{r\rightarrow\infty}
    \tau \sum_{i=1}^r\frac{m}{2}\bigl\|\bm{\nabla} e_p^i \bigr\|^2%
    &\leq  \frac{L_l}{2}\bigl\|e_p^0\bigr\|^2 %
    +  \frac{\tau}{  4\lambda}\norm{e_g^{1}}_\tGamma^2
    \label{mainproofstep6b}
  \end{align}
  follow for for any $r \in \NN$.
Since the right hand sides are independent of $r$, we thereby conclude that the series on the left are absolutely convergent and therefore %
  $ \bigl\| S_l(p_l^n) - \Splni{i-1}\bigr\|$, $\bigl\|\gradepli{i} \bigr\|$ $\longrightarrow 0$ 	 as
$i\rightarrow
  \infty$. Moreover,
  %
%  \footnote{cf. \cite[Theorem A.2.5, p. 252]{Berninger2009}}
  %
  (\ref{mainproofstep6b}) implies $\bigl\|\epli{i} \bigr\|$ $\longrightarrow 0$, $(i\rightarrow \infty)$ as well, by the Poincar{\' e} inequality.

  To show that $\egli{i} \rightarrow 0$ in $\Fs_l'$ we subtract again
  (\ref{weakiterschemeeqn1}) from (\ref{weakiterschemelimiteqn1}) and consider test functions $\varphi_l
\in
  C_0^\infty(\dom_l)$ to get
  \begin{align}
    -\tau \bspl\Fln{l} - \flux{i},\bm{\nabla}\varphi_l\bspr
      &= -L_l\spl \epli{i},\varphi_l \spr %
	+L_l\spl \epli{i-1},\varphi_l \spr %\nonumber  \\%
	%
%	&\hspace{1cm}
- \spl S_l(p_l^n)- S_l(\plni{i-1}), \varphi_l \spr. \label{giconvergencestep1}
  \end{align}
%   \begin{align}
%     \tau \bspl\Fln{l} - \flux{i},\bm{\nabla}\varphi_l\bspr
%       &= \spl -L_l\bigl(\epli{i} - \epli{i-1}\bigr)%
% 	- \bigl(S_l(p_l^n)- S_l(\plni{i-1})\bigr), \varphi_l \spr. \label{giconvergencestep1}
%   \end{align}
  Thus, $\dv\Bigl(\Fln{l} - \flux{i}\Bigr)$ exists in $L^2$ and
  \begin{align}
      -\tau \dv\Bigl(\Fln{l} - \flux{i}\Bigr)
      = L_l \Bigl(\epli{i} - \epli{i-1}\Bigr)  + S_l( p_l^n )- S_l(\plni{i-1})  \label{giconvergencestep2}
  \end{align}
  almost everywhere. Therefore, for any $\varphi_l \in \Fs_l$ one has
  \begin{align}
      \Bigl|\bspl\dv\bigl(\Fln{l}
	- \flux{i}\bigr),\varphi_l \bspr\Bigr|
  %     %
	&\leq \frac{L_l}{\tau}\bigl\|\epli{i} - \epli{i-1} \bigr\|\bigl\|\varphi_l\bigr\|
	%\nonumber \\
	%
%	&
+ \frac{1}{\tau}\bigl\| S_l(p_l^n)- S_l(\plni{i-1})\bigr\| \bigl\| \varphi_l \bigr\|.
      \label{giconvergencestep3}
  \end{align}
  Abbreviating the left hand side of (\ref{giconvergencestep3}) as $\bigl|\Psi_l^{n,i}\bigl(\varphi_l\bigr)\bigr|$,
  (\ref{giconvergencestep3}) means
  \begin{align}
    \sup_{\stackrel{\varphi_l \in \Fs_l}{\varphi_l \neq 0}}
    &\frac{\bigl|\Psi_l^{n,i}\bigl(\varphi_l\bigr)\bigr|}{\norm{\varphi_l}_{\Fs_l}}
    \leq \frac{L_l}{\tau}\bigl\|\epli{i} - \epli{i-1} \bigr\|
%     \nonumber \\
    %
%     &\hspace{1cm}
    + \frac{1}{\tau}\bigl\| S_l(p_l^n)- S_l(\plni{i-1})\bigr\|
    \longrightarrow 0 \quad (i \rightarrow \infty) \label{giconvergencestep4}
  \end{align}
  as a consequence of (\ref{mainproofstep6b}). In other words $\bigl\|\Psi_l^{n,i} \bigr\|_{\Fs_l'} \rightarrow 0$ as
$i
  \rightarrow \infty$.
  Starting again from (\ref{mainproofstep1}) (without the added zero term), this time however
inserting   $\varphi_l \in \Fs_l$, integrating by parts and keeping in mind (\ref{giconvergencestep2}) one gets
  \begin{align}
    \spl \egli{i},\varphi_l \spr_\Gamma
      = - \lambda &\spl\epli{i}, \varphi_l \spr_\Gamma %\nonumber \\
      %
      %&
      +\bspl \bigl[ \Fln{l}- \flux{i} \bigr]\cdot\vt{n_l}, \varphi_l
  \bspr_\Gamma . \label{giconvergencestep5}
  \end{align}
    We already know that $\bigl\|\epli{i}\bigr\|_{\Fs_l} \rightarrow 0$ as $i \rightarrow 0$ so by the
continuity of the trace operator the first term on the right vanishes in the limit.
For the last summand in (\ref{giconvergencestep5}) we use the
integration by parts formula to obtain
  \begin{align}
    \bspl \bigr[ \Fln{l}- \flux{i}  \bigr]\cdot\vt{n_l},\onGamma{\varphi_l}
  \bspr_\Gamma
    &= \Psi_l^{n,i}(\varphi_l)
      + \bspl \Fln{l}- \flux{i}, \bm{\nabla}\varphi_l\bspr.
  \label{giconvergencestep6}
  \end{align}
While the first term on the right approaches 0, the second can be estimated by
  \begin{align}
    \Bigl|\bspl k_l \bigl(S_l(p_l^n)\bigr) \bm{\nabla} \bigl( p_l^n + z \bigr)
    &- k_l\bigl(S_l(p_l^{n,i-1})\bigr)\bm{\nabla} \bigl( p_l^{n,i} + z \bigr)
,\bm{\nabla}\varphi_l\bspr\Bigr|
\nonumber   \\%
%     %
%     &\leq  \bigl\|\bigl(\kln{l} -\klni{i-1}\bigr)\bm{\nabla} \bigl( p_l^n + z \bigr) \bigr\|
%   \bigl\|\bm{\nabla}\varphi_l\bigr\| + \bigl\| \bm{\nabla} p_l^{n,i}\bigr\| \bigl\|\bm{\nabla}\varphi_l \bigr\|
% \nonumber   \\
    %
    &\leq  L_{k_l} M \bigl\|S\bigl(p_l^n\bigr) - S\bigl(p_l^{n,i-1}\bigr)\bigr\| \bigl\|\varphi_l\bigr\|_{\Fs_l} +
\bigl\|   \bm{\nabla} p_l^{n,i}\bigr\| \bigl\|\varphi_l \bigr\|_{\Fs_l},
  \end{align}
  where we used the same reasoning as in (\ref{mainproofstep2b}).
  With this we let $i \rightarrow \infty $ in (\ref{giconvergencestep6}) to obtain
  \begin{align}
    \sup_{\stackrel{\varphi_l \in \Fs_l}{\norm{\varphi_l}_{\Fs_l} = 1}}&
    \Bigl|\bspl \bigr[ \Fln{l}- \flux{i} \bigr]\cdot\vt{n_l},\varphi_l\bspr_\Gamma\Bigr|
    \leq \bigl\|\Psi_l^{n,i}\bigr\|_{\Fs_l'}
    %\nonumber \\
    %
    %&
    + L_{k_l} M \bigl\|S\bigl(p_l^n\bigr) - S\bigl(p_l^{n,i-1}\bigr)\bigr\|  + \bigl\| \bm{\nabla} p_l^{n,i}
\bigr\|
  \longrightarrow 0 . \label{giconvergencestep7}
  \end{align}

  Finally, using the above and letting  $i \rightarrow \infty $ in  (\ref{giconvergencestep5}) gives
  \begin{align*}
    \sup_{\stackrel{\varphi_l \in \Fs_l}{\varphi_l \neq 0}}
    &\frac{\bigl|\spl \egli{i},\varphi_l \spr_\Gamma\bigr|}{\norm{\varphi_l}_{\Fs_l}}
    %
%    \leq \lambda\tilde{C}\bigl\| \epli{i} \bigr\|_{\Fs_l}
 %     + \bigl\|\Psi_l^{n,i}\bigr\|_{\Fs_l'}
      %\nonumber \\
      %&\hspace{1cm}
 %     + L_{k_l} M \bigl\|S\bigl(p_l^n\bigr) - S\bigl(p_l^{n,i-1}\bigr)\bigr\|
  %    + \bigl\| \bm{\nabla} p_l^{n,i} \bigr\|
  \longrightarrow 0 .
%\label{giconvergencestep8}
  \end{align*}
%  where we used the continuity of the trace operator
%  on Lipschitz domains to estimate $\bigl|\spl\epliOnGamma{i},\onGamma{\varphi_l} \spr_\Gamma\bigr| \leq
%  \tilde{C}\bigl\|\epli{i} \bigr\|_{H^1(\dom_l)}\bigl\|\varphi_l\bigr\|_{H^1(\dom_l)}$.
  This shows $\egli{i} \rightarrow 0$ in $\Fs_l'$ for both $l$ and concludes the proof.
\end{proof}

\begin{remark}
  Note that Theorem \ref{mainresult} states that if a solution to the semi-discrete coupled problem exists, then it is the limit of
  the iteration scheme. Since in the convergence proof we use the existence of a solution to
  Problem \ref{WeakSemiDiscreteFormulation}, the argument cannot be used to prove existence. The difficulty lies in the fact that the nonlinearities encountered in the diffusion terms are space dependent and may be discontinuous w.r.t. $x$ over the interface.
%   Instead we mention the possibility of using a Kirchhoff transform (see \cite{AltandLuckhaus}) in each of the subdomains, and reformulate the resulting as two problems defined in each subdomain and coupled through a nonlinear transmission conditions at the common (inner) boundaries. For the existence and uniqueness of solutions to such problems we refer to \cite{Jaeger1998} and \cite{Jaeger2002}, albeit for smoother domains.
%  \marginpar{I must have overlooked somethign but I could not find a relevant existence result in \cite{Berninger2010} or \cite{Berninger2015}, only the equivalence. }
 % More recently, well posedness for the substructured Richards problem is thoroughly discussed %by Berninger et al. in \cite[Chapter 3]{Berninger2009} and
 % in \cite{Berninger2015} as well as \cite{Berninger2010}, where also conditions under
  % {i}which the global and substructured problem are equivalent are being discussed.
\end{remark}

\section{Numerical Experiments} \label{SectionNumericalExperiments}
%\includepdf[pages=1-10]{Numerical_Section-Final.pdf}
% Numerical Section
%
This section is devoted to numerical experiments and the implementation of the proposed domain decomposition
L-scheme.
As our formulation and analysis did not specialise to a
particular spacial discretisation, the numerical implementation of the LDD scheme can in principal be done with finite
difference, finite elements as well as finite volume schemes. Since mass conservation is an essential feature of porous
media flow models, we adopted a cell-centred two point flux approximation variant of a finite volume scheme to reflect
this on the numerical level. The domain $\dom$ is assumed to be rectangular and a rectangular uniform mesh was
used.
\begin{remark}[different decoupling formulations revisited]
  We saw in Remark \ref{RemarkFormulationVariants} that
  another decoupling formulation is possible.
  In fact, this can be taken a step further. Equations (\ref{strongiterschemeeqn2}), (\ref{stronggliupdate}) as well as
  (\ref{strongiterschemeeqn2'}), (\ref{stronggliupdate'}) can be
  embedded into a combined formulation.
  For some  $ 0 < \eta < 1$ and $M > 0 $, consider the generalised decoupling
  \begin{align}
    \flux{i} \cdot \vt{n_l} %
      &=M \left[(1-\eta)g^i_l + \eta p_l^{n,i}\right], \tag{{\ref{strongiterschemeeqn2'}}'}
      \label{strongiterschemeeqn2''}\\%
    (1-\eta)g_l^i&=-2\eta p_{3-l}^{n,i-1}-(1-\eta)g_{3-l}^{i-1}. \tag{\ref{stronggliupdate'}'} \label{stronggliupdate''}
  \end{align}
%   where $\flux{i}$ is the flux in the $i^{\mbox{\scriptsize th}}$ iteration, in our case
%   $ \flux{i} = -k_l\bigl(S_l\bigl(p_l^{n,i-1}\bigr)\bigr) \nabla  \bigl( \vt{ p_l^{n,i}}  - \vt{z} \bigr)$.
  Observe that the $\lambda$-formulation (\ref{strongiterschemeeqn2}), (\ref{stronggliupdate}), as well as the
  convex-combination formulation (\ref{strongiterschemeeqn2'}), (\ref{stronggliupdate'}), are special ca\-ses of this
general
  formulation:
  In particular, $M=(1-\eta)^{-1}$ and $\lambda=\eta (1-\eta)^{-1}$ recovers the $\lambda$-formulation,
  $M=1$ and $\eta = \lambda$ yields the convex-combination formulation.
  Although (\ref{strongiterschemeeqn2''})  and (\ref{stronggliupdate''}) might give even greater parametric control
  over the numerics, in this paper we adhere to the $\lambda$-formulation because of its simplicity. %We choose the $\lambda$-formulation over the convex-combination formulation because the convergence is faster for the former.
  Fig. \ref{figDDnum.lambda} and Fig.  \ref{figDDnum.eta} show the  influence of $\lambda$ and $\eta$ in both
  formulations.
\end{remark}
We start by considering an analytically solvable example. The LDD scheme is tested against other frequently used schemes that do not use a domain
decomposition. All of them are defined on the entire domain and the continuity of normal
flux and pressure over $\Gamma$ is maintained implicitly.
The first scheme to be compared is a finite volume implementation of the original L-scheme on the whole domain (see 
\cite{Slodicka2002,Pop2004365,List2016}), henceforth referred to as LFV scheme. Comparison is also drawn to the modified Picard scheme, (which performs better than the Picard method, see \cite{Celia1990}), which is given by
\begin{align}
  S'_l\bigl(p^{n,i-1}_l\bigr) &\Bigl( p^{n,i}_l-p^{n,i-1}_l \Bigr)%
    +\tau\vt{\nabla}\cdot \flux{i}
%    \nonumber \\ %
   %&
   = \tau f_l -\Bigl( S_l\bigl( p^{n,i-1}_l \bigr) - S_l\bigl( p^{n-1}_l \bigr) \Bigr)  %
   &&\text{ on } \Omega_l,\\
  \bigl\llbracket \flux{i} \cdot \vt{ n_1 } \bigr\rrbracket
     &= 0
    %k_2 \bigl( S_2 (p^{n,i-1}_2) \bigr) \bigl( \vt{\nabla p_l^{n,i}}  - \vt{z} \bigr)\cdot \vt{n_1} %
      &&\text{ on } \Gamma.
\end{align}
Here, the brackets  $\llbracket \cdot \rrbracket$  denote the jump over the interface. Finally, a comparison with the
quadratically convergent Newton scheme is made. Writing
$\delta p^i_l=p^{n,i}_l-p^{n,i-1}_l$, it reads as follows:
\begin{align}
  S'_l &\bigl( p^{n,i-1}_l \bigr)\delta p^i_l
    -\tau\vt{\nabla}\cdot \Bigl[ k_l\bigl( S_l(p^{n,i-1}_l) \bigr) \vt{\nabla} \delta p^i_l
    %\nonumber \\
    %
    %&\hspace{1cm}
    + k'_l \bigl( S_l(p^{n,i-1}_l) \bigr) S'_l \bigl( p^{n,i-1}_l \bigr) \delta p^i_l
\gradPlniPlusGravity{i}
\Bigr]
    \nonumber\\
  &=\tau f_l - \Bigl( S_l \bigl( p^{n,i-1}_l \bigr)-S_l \bigl( p^{n-1}_l \bigr)\!\Bigr)
  %\nonumber \\
  %&\hspace{4mm}
  - \tau \vt{\nabla}\cdot \Bigl(k_l \bigl( S_l(p^{n,i-1}_l)   \bigr) \gradPlniPlusGravity{i-1} \Bigr)
  \quad \text{ on }   \Omega_l \\
  \Bigl\llbracket k_l&\bigl( S_l ( p^{n,i-1}_l ) \bigr) \vt{\nabla} \delta p^i_l\cdot\vt{n_1} \Bigr\rrbracket %
  %\nonumber \\
  %
  %&\hspace{0.5cm}
  + \left\llbracket k_l \bigl( S_l(p^{n,i-1}_l) \bigr)' \delta p^i_l \gradPlniPlusGravity{i-1}
\cdot\vt{n_1}
  \right\rrbracket
  \nonumber \\%
    &=- \left\llbracket k_l \bigl( S_l(p^{n,i-1}_l) \bigr) \gradPlniPlusGravity{i-1} \cdot\vt{n_1} \right\rrbracket%
    \quad\text{ on } \Gamma.%\footnotemark
\end{align}
We refer to \cite{List2016} for a recent study on linearisations for Richards equation.

\begin{figure}[htbp]%{width=0.7\textwidth}
  \centering
  \includegraphics[width=0.6\textwidth]{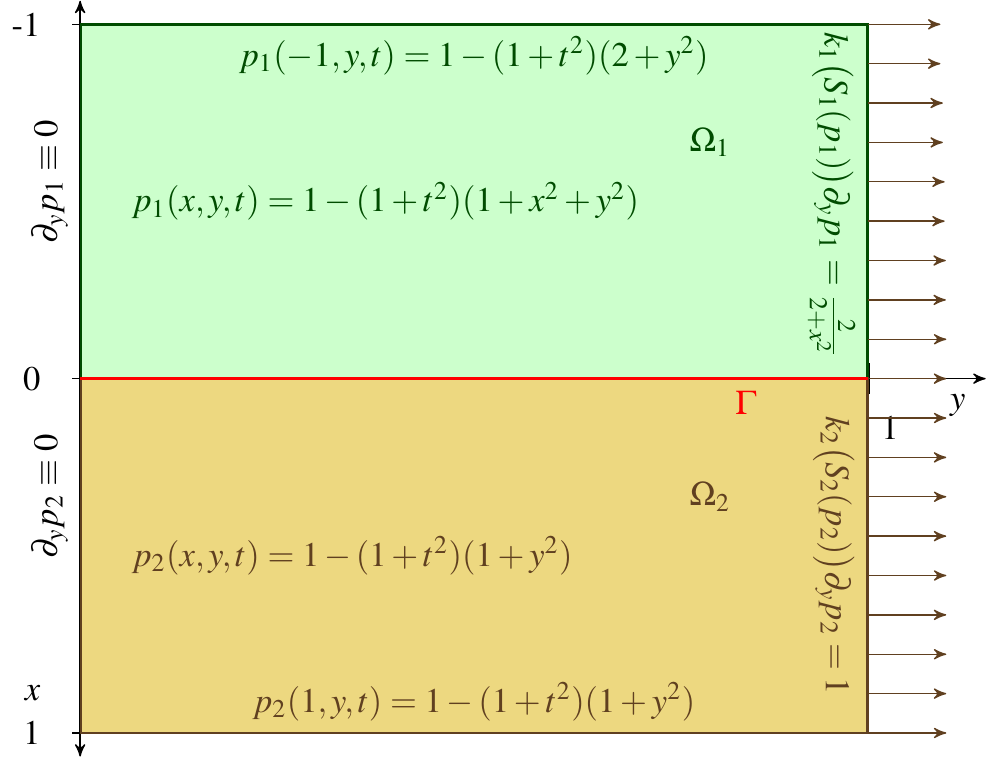}
  \caption{The domain used in the numerical examples. The boundary conditions are given in Table \ref{ICBCtestcase}. The exact solution is also given in each subdomain.}
  \label{numerical_domain}
\end{figure}

\subsection{Results for a case with known exact solution} \label{SubsectionAnaExample}
To demonstrate the robustness of the proposed scheme, we solve
(\ref{ClassicalFormulation1})--({\ref{ClassicalFormulation4}) with both Di\-rich\-let and
Neumann type boundary conditions. In the first case we disregard gravity. Specifically, we consider
\begin{align}
 \dom_1 = (-1,0)\times(0,1), \quad \dom_2 = (0,1)\times(0,1), \quad \text{ and } \Gamma = \{0\} \times [0,1].
\end{align}
%which makes $$.% \medskip
%
The relative permeabilities are $k_1(S_1)= S_1^2$ on $\dom_1$,
$k_2(S_2)= S_2^3$ on $\dom_2$ and the saturations
\begin{align}
  S_l(p)= \begin{cases}
	    \frac{1}{(1-p)^\frac{1}{l+1}} 	&\mbox{ for } p < 0,	\\
	    1 					&\mbox{ for } p \geq 0
	  \end{cases},~\qquad l=1, 2.%
  %
%   &S_2(p)=\begin{cases}
% 	    \frac{1}{(1-p)^\frac{1}{3}} &\hspace{-3mm}\mbox{ for } p < 0	\\
% 	    1 				&\hspace{-3mm}\mbox{ for } p \geq 0
% 	  \end{cases}.
\end{align}%\medskip
The boundaries and right hand sides are chosen to make the exact solution
\begin{equation*}
  \begin{aligned}
    p_1(x,y,t)&=1-(1+t^2)(1+x^2+y^2), 		&& t>0,\, (x,y) \in \dom_1, \\
    p_2(x,y,t)&=1-(1+t^2)(1+y^2),		&& t>0,\, (x,y) \in \dom_2,
  \end{aligned} %\label{AnalyticalSolution}
\end{equation*}
and this corresponds to the right hand sides
\begin{equation*}
  \begin{aligned}
    f_1(x,y,t) &= \frac{4}{(1+x^2+y^2)^2}-\frac{t}{\sqrt{(1+t^2)^3(1+x^2+y^2)}},	\\%&& t>0,\,(x,y) \in \dom_1\\
    f_2(x,y,t) &= \frac{2(1-y^2)}{(1+y^2)^2}-\frac{2t}{3\sqrt[3]{(1+t^2)^4(1+y^2)}},   %&& t>0,\, (x,y) \in \dom_2.
  \end{aligned}
\end{equation*}
for $t>0$, and $(x,y) \in \dom_l$ respectively.
The boundary and initial conditions are summed up in Table \ref{ICBCtestcase}.
%
%
% \begin{figure}[t]
\begin{wrapfigure}[11]{r}{0.47\textwidth}%[ht]
    \vspace{-0.32cm}
    \centering
    \includegraphics[width=0.47\textwidth]{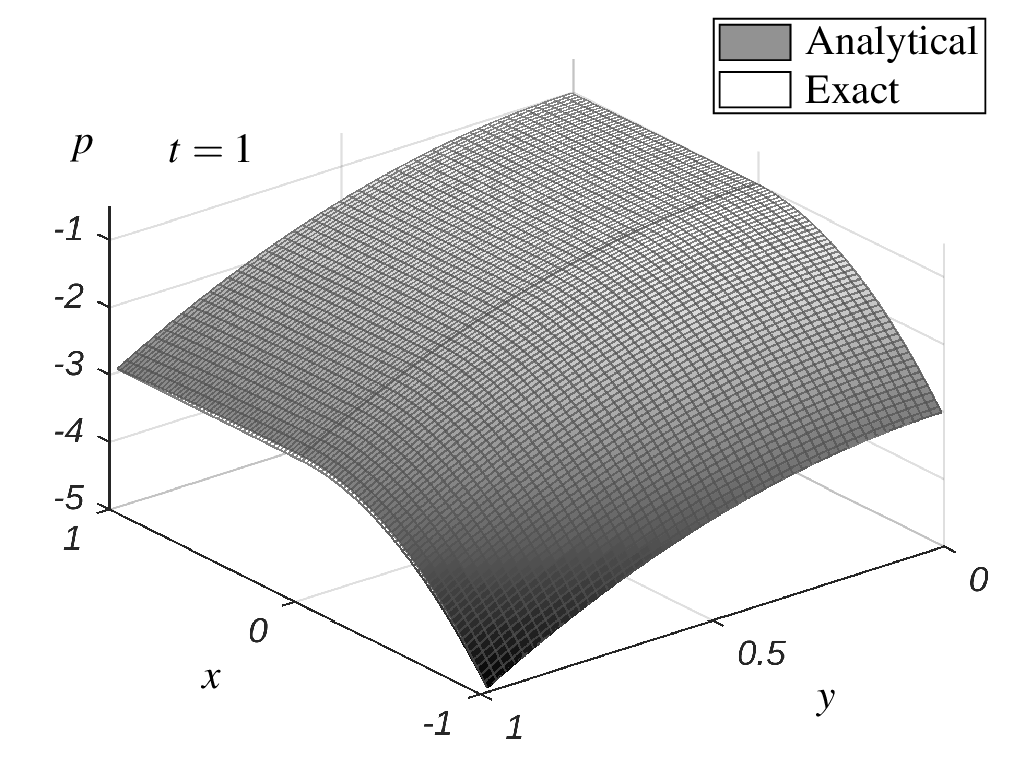}
    \caption{Comparison between the exact pressure and the numerical pressure provided by the LDD scheme.}
    \label{fignum1a}
% \end{figure}
\end{wrapfigure}
%
% \arrayrulewidth=0.6pt
%\tabulinesep=1.2mm
%\taburulecolor{gray!48}
%\newcommand{\tabularwidth}{0.92\textwidth}
%  {\rowcolors{2}{gray!25}{white}
\begin{table}[tbp]%\label{BCandICTableAnaliticTestCase}
    \centering
    %%% 
\begin{tabular}{| c | c | c |}
\hline %&& \\ %to .49\textwidth {X[c,m] X[2.2,c,m] X[2.2,c,m] }%\hline%
  % initial condition
  %\taburowcolors[1]1{black!70 .. white}
  %\multicolumn{3}{c} {\textcolor{white}{\bf Boundary and Initial Conditions for Analytical Case}} \\
  %\rowfont\bfseries%
  %\rowcolor[gray]{0.70}
  %
    & \boldmath $\dom_1$
      & \boldmath$\dom_2$	\\%[0.2em]
\hline %&& \\
  %
  %\taburowcolors[1]2{white .. gray!25}
  $t=0$
  & $p_1(x,y,0)=-(x^2+y^2)$
    &$p_2(x,y,0)=-y^2$	\\
  % boundary condition
  %\rowfont\bfseries%
  %\rowcolor[gray]{0.70}
  BC%
%	& \boldmath $\dom_1$
%	  & \boldmath$\dom_2$	\\
  %
  $y=0$
    &$\partial_y p_1 = 0$
      &$\partial_y p_2 = 0$ \\
  $y=1$
    &$k_1\bigl(S_1(p_1)\bigr)\partial_y p_1=\frac{2}{2+x^2}$
      &$k_2\bigl(S_2(p_2)\bigr)\partial_y p_2=1$\\
  $x=-1$
    &$p_1(-1,y,t)=1-(1+t^2)(2 + y^2)$
      &					\\
  $x=1$
    &
      &$p_2(1,y,t)=1-(1+t^2)(1 + y^2)$\\
%    \tabuphantomline
\hline
\end{tabular}

    \caption{Initial and boundary conditions for the example with exact solution.}
    \label{ICBCtestcase}
  \end{table}
% }
%
%
All linear systems were solved using a restarted generalised minimum residual method (gmres) \cite{Barrett1994}.
%
%\todo{There was a \textbackslash change in the code. What does it mean?}
%
To boost up speed, sparse triplet format was used in the matrix computation. The programs are implemented in ANSI C.
For the implementation we took the same $L_l$ in both sub-domains, i.e. $L:=L_1=L_2$.
%
% \begin{figure}
%   \centering
%   \includegraphics[width=0.7\textwidth]{Profile_Error_new}
%   \caption{Comparison between the numerical solutions provided by the LDD, LFV and the Newton schemes. Plotted are
%   the relative errors %
%   $\bigr\|\frac{p_{\mbox{\tiny exact}}-p_{\mbox{\tiny num}}}{p_{\mbox{\tiny exact}}}\bigr\|$ as
%   functions of $x$, for $y=0.5$ and $t = 1$.}
%   \label{fignum1b}
% \end{figure}
%
% \begin{figure}[thbp]
%   %\vspace*{-0.5cm}
%   \begin{subfigure}[b]{.485\textwidth}
%     \centering
%     \includegraphics[width=\textwidth]{ValidationPt1.png}
%     \caption{Comparison between the exact pressure and the numerical pressure provided by the LDD scheme.}
%     \label{fignum1a}
%   \end{subfigure}~
%   \begin{subfigure}[t]{.485\textwidth}
%     \centering
%     \includegraphics[width=\textwidth]{Profile_Error_new}
%     \caption{Comparison between the numerical solutions provided by the LDD, LFV and the Newton schemes. Plotted are
%    the relative errors %
%    $\bigr\|\frac{p_{\mbox{\tiny exact}}-p_{\mbox{\tiny num}}}{p_{\mbox{\tiny exact}}}\bigr\|$ as
%    functions of $x$, for $y=0.5$ and $t = 1$.}
%     \label{fignum1b}
%   \end{subfigure}
%   \caption{Comparison between the LDD scheme, the exact solution and other schemes (LFV, Newton).}
%   \label{fignum1}
% \end{figure}
% %
%
%
The results are shown in Figures \ref{fignum1a} and \ref{fignum1b}. Fig. \ref{fignum1a} shows the pressure distribution of the exact solution
$p := \chi_{\dom_1}\cdot p_1 + \chi_{\dom_2}\cdot p_2 $ with the numerical solution $p^{n,i} := \chi_{\dom_1}\cdot
p_1^{n,i} + \chi_{\dom_2}\cdot p_2^{n,i} $ plotted on top of it.
For $\Delta x=10^{-2}$, $\Delta t=2\cdot  10^{-4}$ as well as parameters $L=0.25$ and $\lambda =4$, the maximum
relative error was less than $0.03 \%$, i.e. $\bigl\|\tfrac{p^n-p^{n,i}}{p^n}\bigr\|_{L^\infty(\dom)} < 0.0003$.
The relative errors of the LDD, LFV and Newton schemes at the mid-line $y=0.5$ are plotted in Fig. \ref{fignum1b}.
The LDD scheme preserves the flux continuity and pressure continuity at the interface at every time
 step without having to solve for the entire domain. We test this theory numerically.
Fig. \ref{figL2LinftyErrorsDomainAndInterface.analytical} shows how different kinds of errors behave
within one time step. The errors $\norm{p^{n,i} - p^{n,i-1}}_{L^2(\dom)}$, $\norm{p^{n,i} -
p^{n,i-1}}_{L^\infty(\dom)}$   defined on the domain $\dom$, as well as
$\bigl\|\bigl\llbracket p^{n,i} \bigr\rrbracket\bigr\|_{L^2(\Gamma)}$ and $\bigl\|\bigl\llbracket
\vt{F^{n,i}_l}\cdot \vt{n_l} \bigr\rrbracket\bigr\|_{L^2(\Gamma)}$ defined on the interface
$\Gamma$, are shown.
% The convergence $p^{n,i} \rightarrow p^n$, $(i\rightarrow \infty)$ is illustrated in both norms.
We observe that the flux
and pressure jump tend to zero which implies that flux and pressure
continuity is achieved. 
Note that the flux at $x=0$ from the exact solution is $0$.
Next, we compare the LDD scheme with other schemes and study their dependence on discretisation
parameters. We compare the Newton scheme, the (modified) Picard iteration, the already mentioned LFV
scheme and the LDD scheme, investigating the dependence of time step refinement and space grid refinement separately.
\begin{figure}[bhp]
  \begin{subfigure}[t]{0.48\textwidth}
    \centering
    \includegraphics[width=\textwidth]{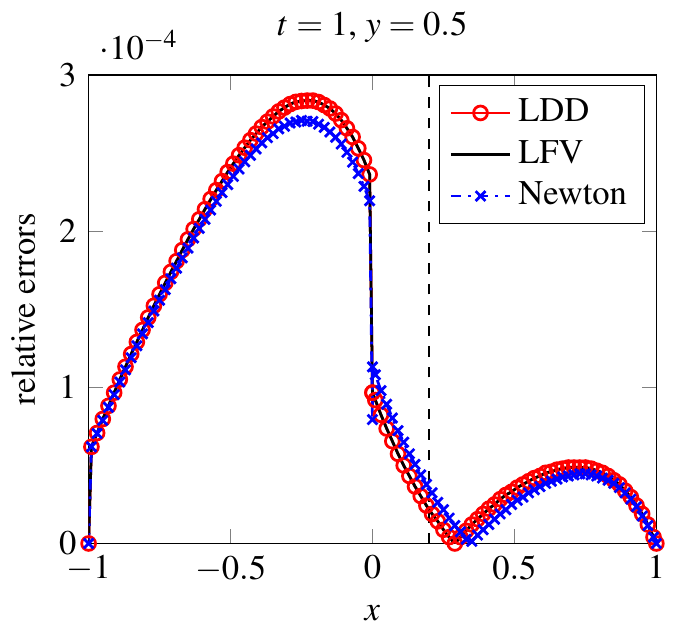}
    \caption{Comparison between the numerical solutions provided by the LDD, LFV and the Newton schemes. Plotted are
    the relative errors %
    $\bigr\|\frac{p_{\mbox{\tiny exact}}-p_{\mbox{\tiny num}}}{p_{\mbox{\tiny exact}}}\bigr\|$ as
    functions of $x$, for $y=0.5$ and $t = 1$.}
    \label{fignum1b}
  \end{subfigure}~ 
  \begin{subfigure}[t]{0.48\textwidth}
    \centering
    \includegraphics[width=\textwidth]{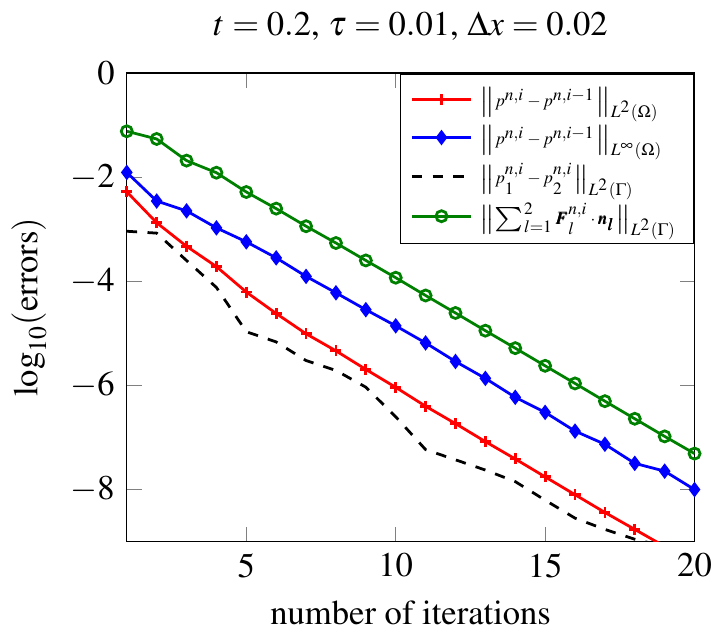}
    \caption{Different errors vs inner iterations for the case with exact solution. Here $t=0.2$, $L=0.25$ and $\lambda=4$.}
    \label{figL2LinftyErrorsDomainAndInterface.analytical}
  \end{subfigure}
\end{figure}

% \begin{wrapfigure}{R}{0.5\textwidth}
%    \centering
%   \includegraphics[width=0.5\textwidth]{Different_Errors_test}
%   \caption{Different errors vs inner iterations for the case with exact solution. Here $t=0.2$, $L=0.25$ and $\lambda=4$.}
%   \label{figL2LinftyErrorsDomainAndInterface.analytical}
% \end{wrapfigure}
% 
% % % % generation of the above graphics

The first study, shown in Fig. \ref{fig_DDnum.mesh}, plots $\log_{10}\bigl(\|p^{n,i}
-p^{n,i-1}\|_{L^2(\dom)}\bigr)$ for all schemes, at the fixed time step corresponding to $t=0.2$. As expected, Newton is
the fastest and shows a quadratic convergence rate. But at the same time, it is most susceptible to change in mesh size
as observed from the slopes of the left-most curves. The convergence rate of the Picard iteration is linear, faster
than
both the $L$-schemes and is stable with respect to variation in mesh size. The  L-schemes also exhibit linear
convergence, albeit slower than Picard, and the
convergence speed does not vary much with mesh size. LFV and LDD schemes have practically the same convergence rate.
Table \ref{speedcomparisontable} complements the plot in Fig. \ref{fig_DDnum.mesh} and
lists experimental average convergence rates, defined as $\|e^{n,i+1}_p\|/\|e^{n,i}_p\|$, for all
schemes (Newton data is not shown for $\Delta x=0.1$, $0.05$, $0.02$ as it reaches an error lower
than $10^{-10}$ in 3 iterations).
%
%\vspace{-2cm}
\begin{figure}[tbp]
  \centering
  \includegraphics[width=0.94\textwidth]{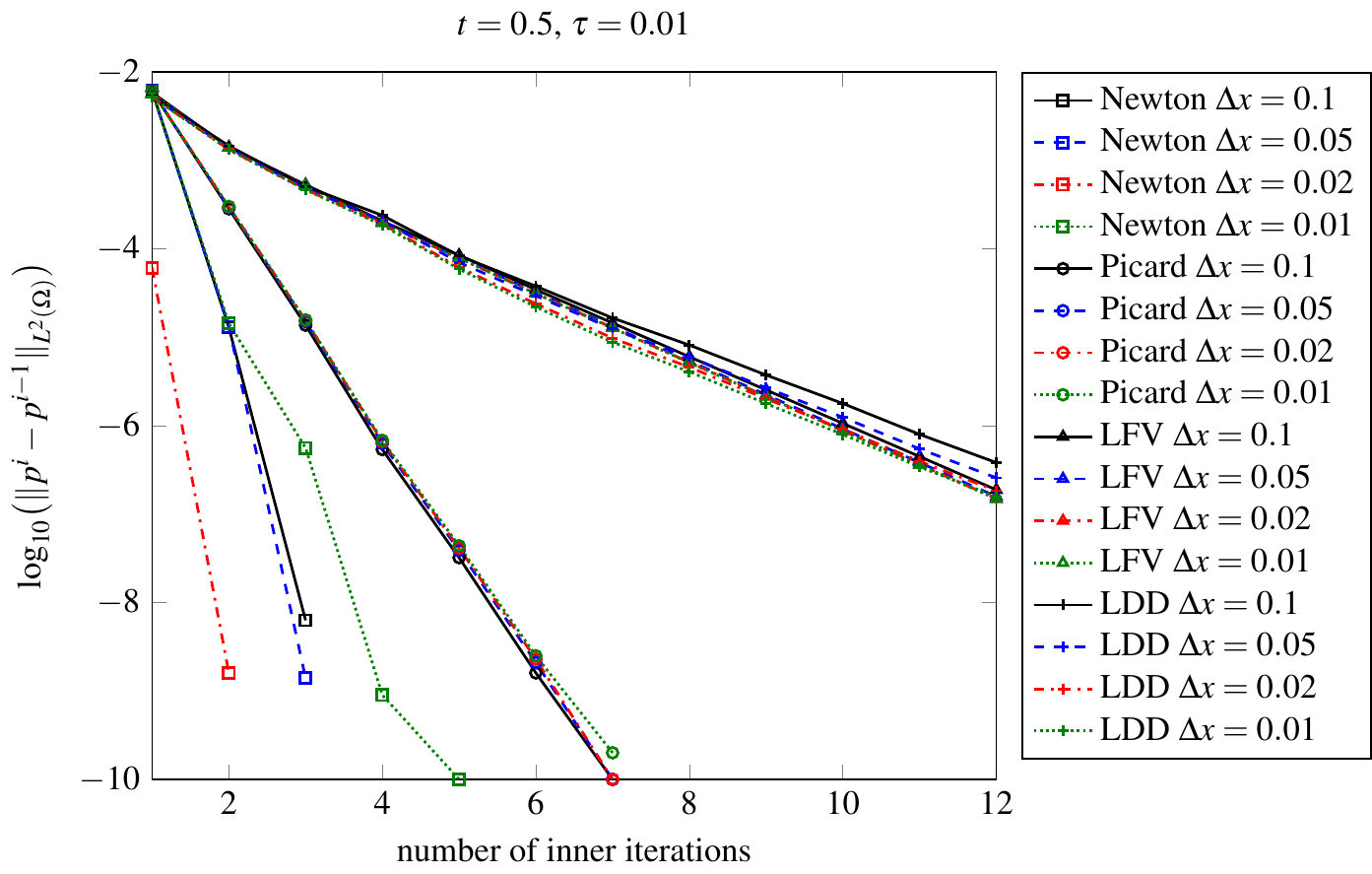}
  \caption{Performance comparison and mesh study for the convergence of the LDD, LFV, Picard  and Newton schemes. Here $L=0.25$
and $\lambda=4$.}
  \label{fig_DDnum.mesh}
\end{figure}

\begin{wraptable}{R}{0.5\textwidth}%[hbtp]
  \centering
%  \tabulinesep=1.2mm
%  \taburulecolor{gray!48}
%  \taburowcolors[1]1{black!70 .. black!70}
%     \begin{tabular}[width=0.45\textwidth]{c c c c c}
  %%%% SPEED COMPARISON TABLE
\begin{tabular}{| c | c | c | c | c|} %to 0.48\textwidth {X[c,m] X[c,m] X[c,m] X[c,m] X[c,m]}%\hline%
\hline %&&&&  \\ %\textcolor{white}{$\Delta x$} &
    %\multicolumn4{c}{\textcolor{white}{Average convergence rates $\|e^{n,i+1}\| / \|e^{n,i}\|$}} &  \\
    %\rowcolor[gray]{0.7}
  $\Delta x$ &\textcolor{black}{Newton} &\textcolor{black}{Picard}
  &\textcolor{black}{LFV}
  &\textcolor{black}{LDD} \\
  \hline
  %\taburowcolors[2]2{gray!25 .. white}
  $0.1$  &- &0.0504 &0.4046 & 0.4400\\
  $0.05$ &- &0.0504 &0.3906 &0.4270\\
  $0.02$ &- &0.0505 &0.3909 &0.4221\\
  $0.01$ &   0.0113 &0.0567 &0.3910 &0.4221\\
  %\rowcolor[gray]{0.7}
  Type &Quadratic &Linear &Linear &Linear \\
  %\tabuphantomline
  \hline
\end{tabular}

  \caption{The average convergence rate, $\|e^{n,i+1}\| / \|e^{n,i}\|$, for the different schemes and with respect
  to the mesh-size.}\label{speedcomparisontable}
\end{wraptable}
Secondly, we study the dependence of the convergence rates on time step size for a fixed mesh size
($\Delta x = 0.02$). The error characteristics of all four schemes in Fig. \ref{fig_DDnum.timestudy} are
shown for $t=0.5$. In Fig. \ref{fig_DDnum.timestudy1} both, Newton and Picard, diverge, whereas both
$L$-schemes converge for $L=0.25$. The LFV scheme exhibits some oscillations, the  reason  being the
dependence of the choice of $L$ on the time step $\tau$. Higher values of $\tau$ might require higher values
of $L$. Indeed, if we substitute $L=0.5$ in the LFV scheme (marked as LFV* in the diagram), we see a
more robust behaviour. Note, that the LDD scheme converges for all $\tau$ and is at least as fast as
the LFV scheme in all the cases.
For smaller values of $\tau$ the Newton and Picard iteration converge faster than both L-schemes, as shown in Figures \ref{fig_DDnum.timestudy2} and \ref{fig_DDnum.timestudy3}.
\begin{figure}[bthp]
  \begin{subfigure}[b]{.48\textwidth}
  \centering
  \includegraphics[width=\textwidth]{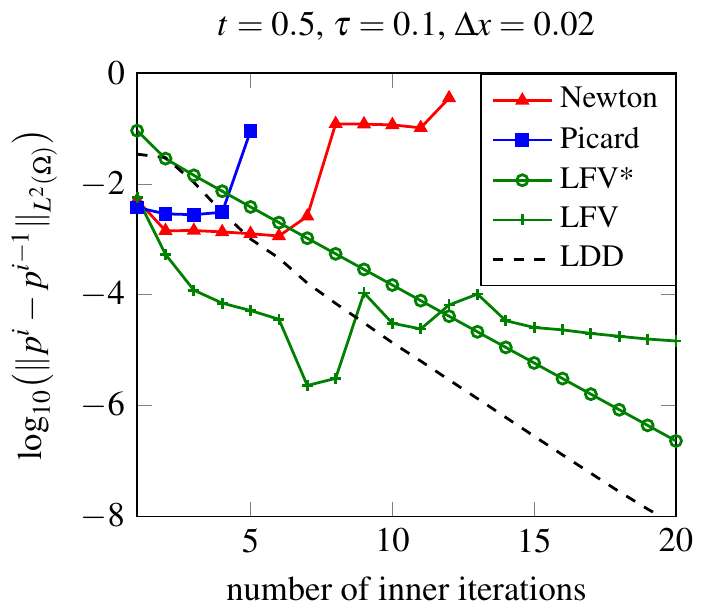}
  \caption{$\Delta t =0.1$}
  \label{fig_DDnum.timestudy1}
\end{subfigure}~
\begin{subfigure}[b]{.48\textwidth}
  \centering
  \includegraphics[width=\textwidth]{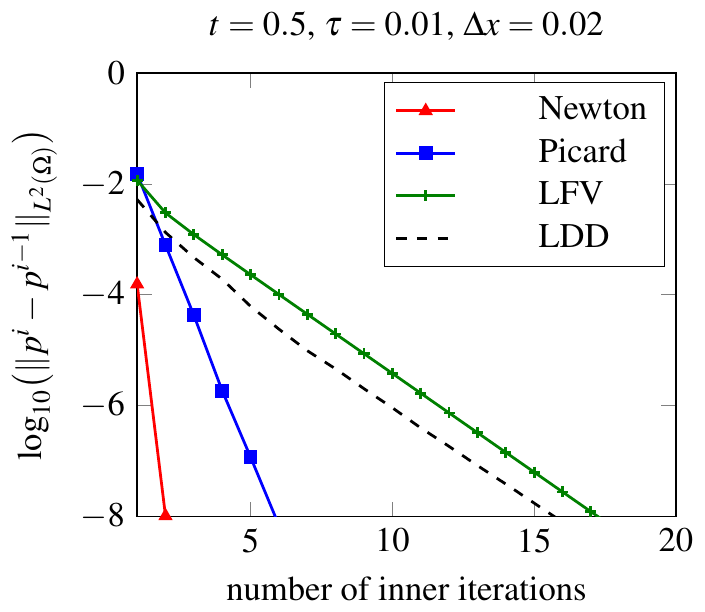}
  \caption{$\Delta t =0.01$}
  \label{fig_DDnum.timestudy2}
\end{subfigure} %\\
\centering
\begin{subfigure}[b]{.48\textwidth}
  \centering
  \includegraphics[width=\textwidth]{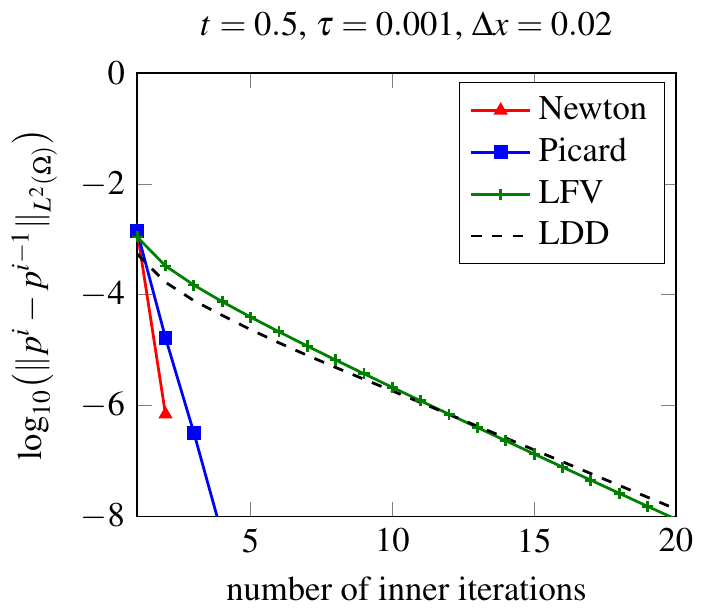}
  \caption{$\Delta t =0.001$}
  \label{fig_DDnum.timestudy3}
\end{subfigure}
  \caption{Convergence study for the time-steps $\tau=0.1$, $0.01$, $0.001$. Here, $L=0.25$ for the LFV scheme and $L=0.5$
  for the LFV* scheme. For the LDD scheme one has $L=0.25$, $\lambda=2$ in case \ref{fig_DDnum.timestudy1}, $L=0.25$, $\lambda=4$ in case
  \ref{fig_DDnum.timestudy2}, and $L=0.25$, $\lambda=10$ in case \ref{fig_DDnum.timestudy3}.}
  \label{fig_DDnum.timestudy}
\end{figure}
% \begin{figure}[htbp]
%  \centering
% %\begin{subfigure}[b]{.5\textwidth}
%   \centering
%   \includegraphics[width=.5\textwidth]{dt0p001}
%   \caption{$\Delta t =0.001$}
%   \label{fig_DDnum.timestudy3}
% % \end{subfigure}
% %   \caption{Convergence study for the time-steps $\tau=0.1$, $0.01$, $0.001$. Here, $L=0.25$ for the LFV scheme and $L=0.5$
% %   for the LFV* scheme. For the LDD scheme one has $L=0.25$, $\lambda=2$ in case \ref{fig_DDnum.timestudy1}, $L=0.25$, $\lambda=4$ in case
% %   \ref{fig_DDnum.timestudy2}, and $L=0.25$, $\lambda=10$ in case \ref{fig_DDnum.timestudy3}.}
% \end{figure}
%
\begin{wrapfigure}[13]{R}{0.5\textwidth}
% \begin{figure}[bthp]
%   \vspace{-0.4cm}
  \includegraphics[width=0.49\textwidth]{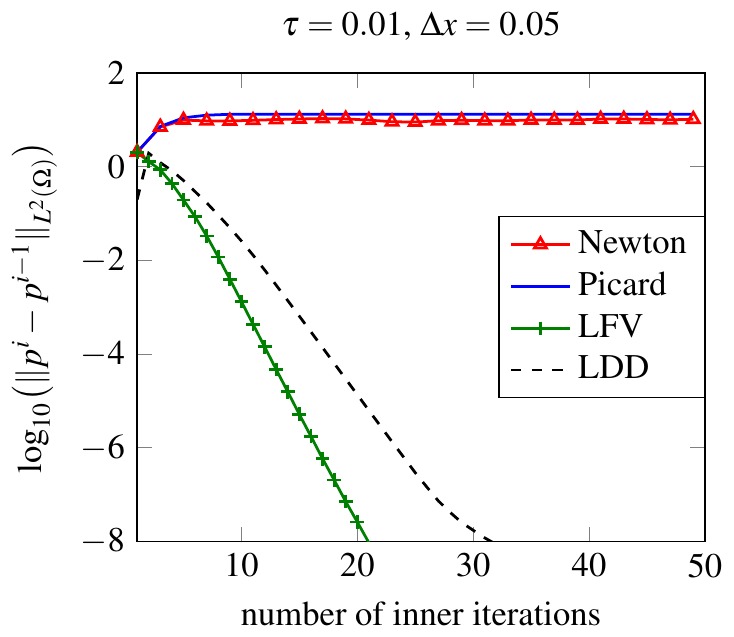}
  \caption{Error decay for the different schemes for a constant initial guess, $p^{n,0}=-5$. Here $L=0.25$, $\lambda=4$.}
  \label{fig_DDnum.ConstantIniGuess}
\end{wrapfigure}

According to the theory, the convergence of the Newton and Picard schemes is only guaranteed if the initial guess is close enough to the exact solution. Therefore, starting the iteration with the numerical solution at the previous time step this suggests that the time step should be taken small enough to have a guaranteed convergence (see \cite{Park1995,Radu2006,List2016}. Contrariwise, L-schemes are free of this constraint.

To illustrate this behaviour, we have investigated the convergence of the schemes for a constant initial guess. Specifically, $p^{n,0}=-5$ has been used instead of $p^{n,0}=p^{n-1}$. In this case, the Newton and Picard schemes are divergent whereas both L-schemes still produce a good approximation after several iterations. This is displayed in Fig. \ref{fig_DDnum.ConstantIniGuess}. A similar behaviour will be observed again while discussing a numerical example with realistic parameters.

\begin{remark} \label{RemarkOnLambdaDependece}
  The convergence behaviour of the LDD scheme can be optimized by choosing $\lambda$ properly. In the above
  comparison $\lambda$ was chosen differently for every choice of mesh size. The optimality of $\lambda$ is dependent
  on the mesh and the time step size. With a good choice of $\lambda$, one can make the LDD scheme at least as fast as the LFV
  scheme. This is discussed in more detail in Section \ref{SubsectionParamDependenceAndKeyFeatures}.
\end{remark}

\subsubsection{Results for a realistic case with van Genuchten parameters}\label{SectionRealisticCase}

We demonstrate the applicability of the LDD scheme for a case with realistic parameters, incorporating also gravity effects. We consider a van-Genuchten-Mualem parametrisation \cite{VanGenuchten1980} with the curves $k$ and $S$
  \begin{equation}
    \begin{aligned}
      S_l(p)    &= S_{l,r}+(S_{l,s}-S_{l,r})\Phi_l(p),                          \\
      \Phi_l(p) &= \dfrac{1}{\bigl(1+(-\alpha_l p)^{\hat{n}_l}\bigr)^{m_l}}, \qquad                    
m_l=1-\dfrac{1}{\hat{n}_l},\\
      k_l(S)    &= \sqrt{\Phi_l(p)}\Bigl(1-\bigl(1-\Phi_l(p)^{\frac{1}{m_l}}\bigr)^{m_l}\Bigr)^2.   
    \end{aligned}
  \end{equation}
  The specific parameter values are listed in
  Table \ref{RealisticTestCaseParameters} and are characteristic for particular types of materials, {\itshape  silt loam
  G.E. 3} ($\Omega_1$) and {\itshape sandstone} ($\Omega_2$). These materials have very different absolute permeabilities $\kappa_1, \kappa_2$, which makes the numerical calculations more challenging.

The dimensional governing equations and boundary
  conditions  become ($l=1,2$)
  \begin{align}
    L_l p_l^{n,i}+\tau\nabla&\cdot\flux{i}
    =L_l p_\ell^{n,i-1}
    \nonumber \\
    &- \phi_l \bigl( S_l(p_l^{n,i-1})- S_l(p_l^{n-1}) \bigr), %
        && \text{ on } \Omega_l, \\
    \flux{i}\cdot \vt{n_l}%
      &= g^i_l + 2\lambda p_l^{n,i}, %
        &&\text{ on } \Gamma,\\
     p_l^{n,i},
     &=0 %
       &&\text{ on } \partial\Omega_l.
\end{align}
In this case $\flux{i} = -\frac{\kappa_l}{\mu}k_{r,l} \bigl( S_l(p_l^{n,i-1}) \bigr) \bigl(\bm{\nabla} p_l^{n,i}
-\rho\vt{g}\bigl)$. Here $\vt{g} = g\vt{e_x}$ is the gravitational acceleration aligned with the positive $x$-direction,
$\rho$, $\mu$ are the density and the viscosity of the fluid and $\kappa_l$, $\phi_l$ are the absolute permeability as well as
the porosity of the medium. Note that Fig.
\ref{numerical_domain} is rotated by 90 degrees.
% %
% The parameters for these porous materials have been taken from \cite{VanGenuchten1980} and are:
% cf. \cite{VanGenuchten1980}.
%\Cref{RealisticTestCaseParameters} below lists the parameters used for the simulation of this example.
\begin{table}[hbtp]
  \centering
  \arrayrulewidth=0.6pt
  %\tabulinesep=1.5mm
  %\taburulecolor{gray!48}
  %\newcommand{\tabularwidth}{0.92\textwidth}
%   \rowcolors{2}{gray!25}{white}
  %\taburowcolors[1]1{black!70 .. white}
  % \begin{figure}[htbp]
  %   \centering
  %%%%% VAN-GENUCHTEN-MUALEM PARAMETER TABLE
\begin{tabular}{| c | c | c | c |} %to 0.48\textwidth {X[c,m] X[c,m] X[c,m] X[c,m] X[c,m]}%\hline%
 \hline %&&&   \\  %
    %\multicolumn{4}{c}{\textcolor{white}{\bfseries Parameters for the Van-Genuchten-Mualem Parametrisation}} \\
    %\rowfont\bfseries%
    %\rowcolor[gray]{0.70}
    Parameter 	&  Unit		& Silt Loam G.E. 3  $(\dom_1) $	
      & Sandstone $(\dom_2)$	\\
      \hline
%     %
    %\taburowcolors[3]1{white .. gray!25}
    Porosity ($\phi_l$)	& - & $0.35$ 	&    $0.35$ 		\\
%     %
    Water Density ($\rho$) & kg m$^{-3}$ 	& $1\times10^{3}$ 	&$1\times10^{3}$ \\
    Water Viscosity ($\mu$) 			& Pa$\cdot$s 	& $1\times10^{-3}$ 	&
    $1\times10^{-3}$ 	\\
    Absolute permeability ($\kappa_l$)		& m$^2$s  	& $5.7407\times10^{-14}$ 	&
    $1.2500\times10^{-12}$ 	\\
    Retention exponent ($\hat{n}_l$) 			& - 		& $2.06$ 		&
    $10.4$ 		\\
    Retention parameter ($\alpha_l$) 		& Pa$^{-1}$ 	& $4.23\times10^{-5}$ 	&
    $7.90\times10^{-5}$ 	\\
    Irreducible water saturation ($S_{l,r}$) 	& - 		& $0.131$ 		&
    $0.153$ 		\\
    Irreducible air saturation ($1-S_{l,s}$)   	& - 		& $0.604$ 		&
    $0.75$ 		\\
% %    \tabuphantomline
\hline
\end{tabular}

  \caption{The van Genuchten-Mualem parameters in the realistic test case.}
  \label{RealisticTestCaseParameters}
\end{table}
The problem is nondimensionalised by using the characteristic pressure $p^*:=-14.8\times10^3$Pa, length $1.48$m and
time $41.440$s. This leads to the nondimensional quantities $\tilde{p}$, $(x,y)$ and $t$.  After nondimensionalisation, the domain used is again taken to be $\Omega_1=(-1,0)\times(0,1)$, $\Omega_2=(0,1)\times(0,1)$. 
% \todox{What is meant by this. Is the domain in nondimensional coordinates like in fig 2 or in old coordinates? what is a-dimensionalisation?}
The initial condition used is
\begin{align}
 \tilde{p}(x,y,0)=-1
\end{align}
and boundary conditions are
\begin{align*}
 &\tilde{p}(-1,y,t)=
 \begin{cases}
 -1+ty &\text{ if } y<(1-\epsilon)t^{-1}\\
 -\epsilon &\text{ if } y\ge(1-\epsilon)t^{-1}
 \end{cases},
 \\
&\tilde{p}(1,y,t)=-1,
\end{align*}
together with a no-flow condition at $y=0, 1$.
We take $\epsilon > 0 $ to avoid degeneracy. %Here, $\epsilon = 0.01$ is used.

% \begin{figure}[htbp]
%   \centering
%   \includegraphics[width=0.49\textwidth]{Different_Errors_real}
%   \caption{Different errors vs inner iterations for the realistic case at $t=0.2$. The parameters are $\tau=0.01$, $\Delta
% x=0.02$, $L_l=0.25$ and
%   $\lambda=10$. Only the LDD scheme is shown in this
%   plot.}
%   \label{fig_DDnum.diff.err.real}
% \end{figure}
% \begin{figure}[hbtp]
%   \centering
%   \includegraphics[width=0.49\textwidth]{Real_different_schemes}
%   \caption{Error vs inner iterations for the realistic case. The LDD, LFV and Newton errors are plotted at $t=0.2$.
%     Newton$^*$ denotes the error of the Newton scheme at $t=0.9$. Picard is plotted at $t=0.02$.  Here, $L=0.5$,
%     $\lambda=10$.}	
%     \label{fig_DDnum.diff.scheme.real}
% \end{figure}
%
\begin{figure}[h]
  \begin{subfigure}[b]{.49\textwidth}
    \centering
    \includegraphics[width=\textwidth]{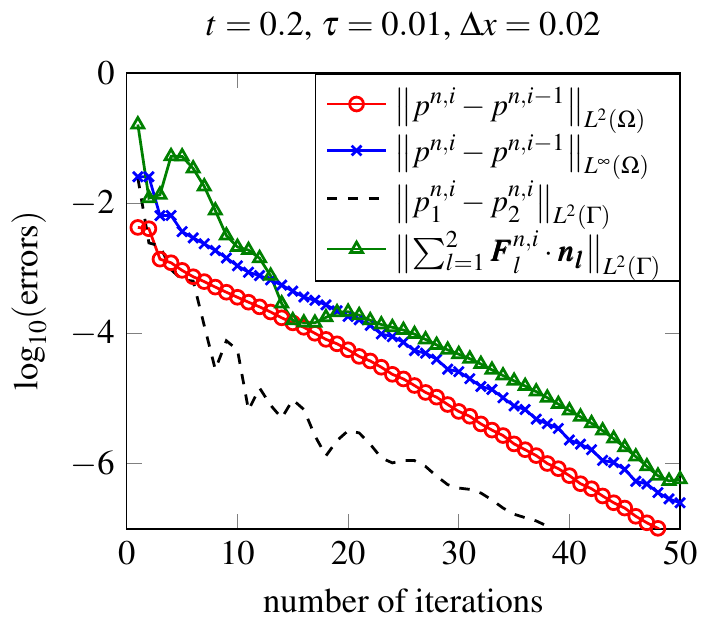}
    \caption{Different errors vs inner iterations for the realistic case at $t=0.2$. The parameters are $\tau=0.01$, $\Delta
x=0.02$, $L_l=0.25$ and
  $\lambda=10$. Only the LDD scheme is shown in this
  plot.}
    \label{fig_DDnum.diff.err.real}
  \end{subfigure}~
  \begin{subfigure}[b]{.49\textwidth}
    \centering
    \includegraphics[width=\textwidth]{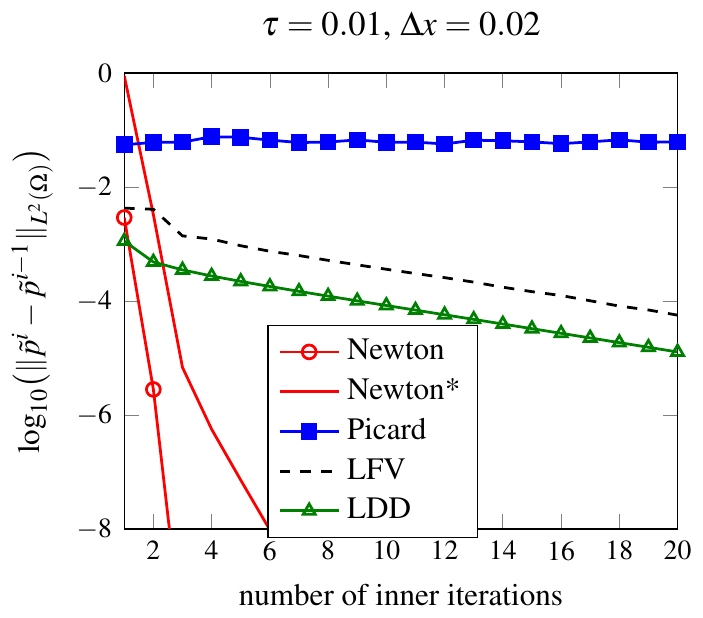}
    \caption{Error vs inner iterations for the realistic case. LDD , LFV and Newton errors are plotted at $t=0.2$.
    Newton$^*$ denotes the errors of Newton scheme at $t=0.9$. Picard is plotted at $t=0.02$.  Here, $L=0.5$,
    $\lambda=10$.}	
    \label{fig_DDnum.diff.scheme.real}
  \end{subfigure}
  \caption{Error plots and scheme comparison for the realistic case.}
  \label{fig_DDnum.real}
\end{figure}
Fig. \ref{fig_DDnum.diff.err.real} shows the different errors for this case and it can be seen that all the errors
are decreasing for the LDD scheme. Errors at the interface and inside the domain tend to $0$, the convergence is slower
compared to the case with exact solution, however. This is due to the large variance of the parameters as well as the highly
nonlinear nature of the associated functions. Because of this, both Newton and Picard schemes diverge. The behaviour
of
different schemes for the same set of parameters is shown in Fig. \ref{fig_DDnum.diff.scheme.real}. Observe that for
the
Newton scheme the starting error as well as the number of iterations required increases steadily with $t$ until
$t=0.94$, at which point the errors start diverging.
The Picard scheme becomes divergent even before $t=0.2$. In contrast, both L-schemes remain stable in this case.
\subsection{Time Performance} \label{SubsectionTimePerformance}
This section is devoted to the comparison of time performance of the schemes.  We have seen that L-schemes are more
stable than  Newton and Picard. But if they are converging, then Newton and Picard schemes converge faster than the
L-schemes. Below we investigate how the schemes compare to one another with respect to \emph{actual computational
time}. We set an error tolerance for the schemes that stops the iterations within one time step, after reaching an
error lower than $10^{-6}$, i.e. $\norm{p^{n,i}-p^{n,i-1}}_{L^2(\dom)} < 10^{-6}$. This is to ensure that we get comparative
CPU-clock-time for different schemes for the same degree of accuracy.

\begin{wraptable}{R}{0.49\textwidth}
%   \vspace{-0.6cm}
  \centering
  %\rowcolors{2}{gray!25}{white}
  %\taburowcolors[1]1{black!70 .. white}
  %%%% CONDITION NUMBERS
\begin{tabular}{| c | c | c | c |} %to 0.48\textwidth {X[c,m] X[c,m] X[c,m] X[c,m] X[c,m]}%\hline%
 %\hline %&&&   \\  %
 %  \begin{tabu} to 0.47\textwidth {X[2,c,m] X[c,m] X[c,m] X[c,m]}
  %  \multicolumn{4}{c}{\color{white}\bfseries Condition numbers for LDD and LFV schemes} \\
   % \rowcolor[gray]{0.7}
   \multicolumn{4}{c}{Condition number}\\
   \hline
%    \rowcolor[gray]{0.4}
    {$\Delta x$} &{0.1} &{0.05} &{0.02}\\
    \hline
    %\taburowcolors[3]1{gray!25 .. white}
    L-DD ($\Omega_1$) &7.6191  &11.8947 &73.362 \\
    L-DD ($\Omega_2$) &7.0219  &12.3557 &74.519 \\
    L-FV ($\Omega$) 	&94.8158 &171.47  &397.34 \\
  \hline
\end{tabular}

  \caption{The condition number vs mesh size for the LDD and LFV schemes. Here, $\tau=0.001$, $t=0.2$, $L=0.25$, $\lambda=10$.
  The condition numbers are calculated for the $200^{th}$ time step for the matrices of the first inner
  iteration.}
  \label{DDnum.cond.table}
\end{wraptable}

We computed the exactly solvable case on a LINUX server ({mammoth.win.tue.nl}) for all four schemes using the same set
of parameters ($\Delta x=0.02$, $\tau=0.001$, $L=0.25$ and $\lambda=10$). Figure \ref{fig:DDnum.timea}  illustrates the
time-performances of these schemes over the whole computational time domain. Table \ref{fig:DDnum.timeb} shows how many inner-iterations are required on average for different schemes to reach the error criterion at different points in
time. 

Iteration requirement per time step increases for all schemes as the boundary conditions change more rapidly
with time. Table \ref{fig:DDnum.timeb} shows the average time taken and how many gmres iterations (outer and inner)
were required by each scheme to execute one inner iteration.

Unsurprisingly, the Newton scheme is still fastest, followed by Picard and the LDD scheme. But LDD competes closely
with Newton and Picard. Even more surprising is the fact that the LFV scheme takes

\begin{wrapfigure}{R}{.5\textwidth}
%   \vspace{-0.5cm}
    \centering
    \includegraphics[width=0.49\textwidth]{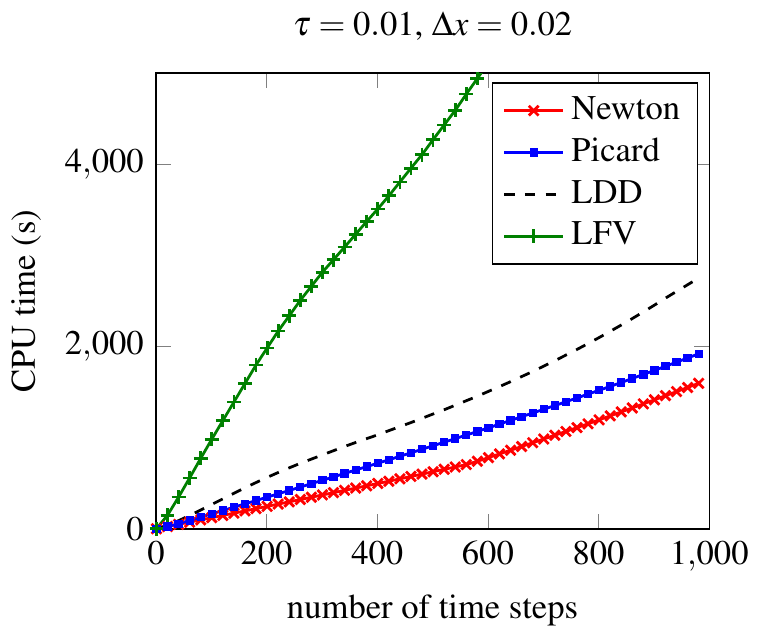}
    \caption{Time performance of the L-DD, L-FV and the Newton-FV schemes.}
    \label{fig:DDnum.timea}
\end{wrapfigure}
considerably more time to reach the
desired accuracy compared
to the LDD scheme, despite both having almost the same convergence rate. The reason becomes apparent from
Table \ref{fig:DDnum.timeb}: The LDD scheme requires much less time \emph{per inner iteration} than all other schemes.
The LFV scheme has the  second fastest average time per iteration. For the Picard iteration, the derivative of the saturation
function needs to be evaluated   which in turn costs more time than an iteration in the LFV scheme. The Newton scheme
is computationally most expensive per iteration because it calculates the Jacobian at every iteration.
\begin{table}[bpt]
  \centering
%  \taburowcolors[1]1{black!70 .. white}
  %\rowcolors{2}{gray!25}{white}
  %%%% AVERAGE ITERATION PER TIME STEP
\begin{tabular}{| c | c | c | c | c |} %to 0.48\textwidth {X[c,m] X[c,m] X[c,m] X[c,m] X[c,m]}%\hline%
 %\hline %&&&&   \\  %
   %\begin{tabular} to 0.48\textwidth {X[2.6,c,m] X[c,m] X[c,m] X[c,m] X[c,m]}%\hline%
    %\multicolumn{5}{c}{\textcolor{white}{Average inner iterations required}}\\
    \multicolumn{5}{c}{Average inner iterations required}\\
 \hline
    %\rowcolor[gray]{0.7}
    Time-step/Scheme & {LDD} & {LFV} & {Picard} &
    {Newton}\\
    %\taburowcolors[1]1{gray!25 .. white}
    10 	&7.3 	&7.5 	&2.3 	&2.3\\
    50 	&9.880 	&10.72 	&2.520 	&2.060\\
    100 	&11.26 	&12.31 	&2.760 	&2.030\\
    500 	&11.19 	&11.79 	&2.952 	&2.006\\
    1000 	&14.18 	&14.35 	&2.946 	&2.408\\
    %\rowcolor[gray]{0.7}
    Avg. time per iter. 	& 0.1965 	& 0.5392	& 0.6591 & 0.6722\\
    %\rowcolor[gray]{0.7}
    Avg. GMRES iterations	& 119+ 123 	& 396.6 	& 390.9  & 397.7 \\
    %\tabuphantomline
    \hline
\end{tabular}

  \caption{ The average number of inner iterations per time step required  by the different schemes to reach the stopping criterion  $\norm{p^{n,i}-p^{n,i-1}}_{L^2(\dom)} < 10^{-6}$. The last two rows give the average time and gmres-iterations per inner iteration.}
  \label{fig:DDnum.timeb}
\end{table}

The schemes that do not decouple the domain require much more time  and many more gmres-iterations per inner iteration.
The reason is that the domain decomposition schemes involve smaller matrices and and they have smaller condition
numbers. This is illustrated by the last row of Table \ref{fig:DDnum.timeb}. The LDD scheme requires on  average 119
gmres-iterations  on $\dom_1$ and 123 gmres-iterations on $\dom_2$ and both domains have $52\times 50$ elements. Compare this with Newton, which takes almost 400 gmres-iterations and deals with $104\times 50$ variables on each gmres-iteration. This explains
why the LDD scheme takes so much less time per inner iteration. Table \ref{DDnum.cond.table} compares the condition
numbers of the LDD and the LFV scheme. It shows that the matrices for the LFV scheme are worse conditioned than the ones of the LDD
scheme. The latter has two condition numbers, one for each domain. The 2-norm condition numbers were calculated with
 MATLAB's build in cond() function.

\begin{remark}
 The fact that the LDD scheme performance competes closely with Newton and Picard, means that, LDD can potentially
be made much   faster than even Newton as it is parallelisable. This is the key advantage of the LDD scheme along with its
global convergence property.
\end{remark}
\subsection{Parameter dependence and key features} \label{SubsectionParamDependenceAndKeyFeatures}

Having outlined the robustness and speed of the proposed LDD scheme we turn to investigate some of its properties. Two
important parameters have been introduced in the L-DD scheme, i.e. $L_l$ and $\lambda$, and apart from
a lower bound on $L_l$ nothing has been specified about these parameters. This means that they can freely be adjusted
to give optimal convergence rate. In fact, in this section we will see that the convergence rate depends strongly on
these parameters.

\subsubsection*{The influence of \texorpdfstring{$\lambda$}{lambda}}

\begin{figure*}[tp]
  \begin{subfigure}[b]{.49\textwidth}
    \centering
    \includegraphics[width=\textwidth]{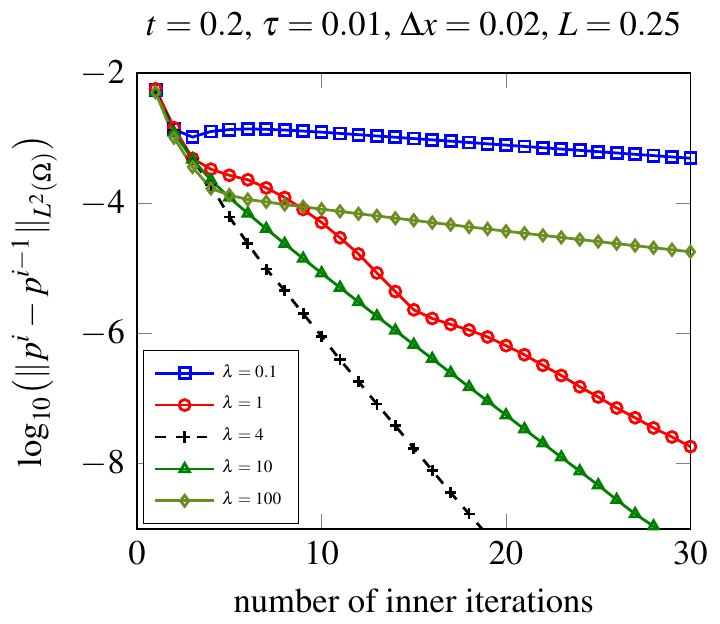}
    \caption{The decay of the pressure error in terms of  $\lambda$.}
    \label{figDDnum.lambda.err}
  \end{subfigure}~
  \begin{subfigure}[b]{.49\textwidth}
    \centering
    \includegraphics[width=\textwidth]{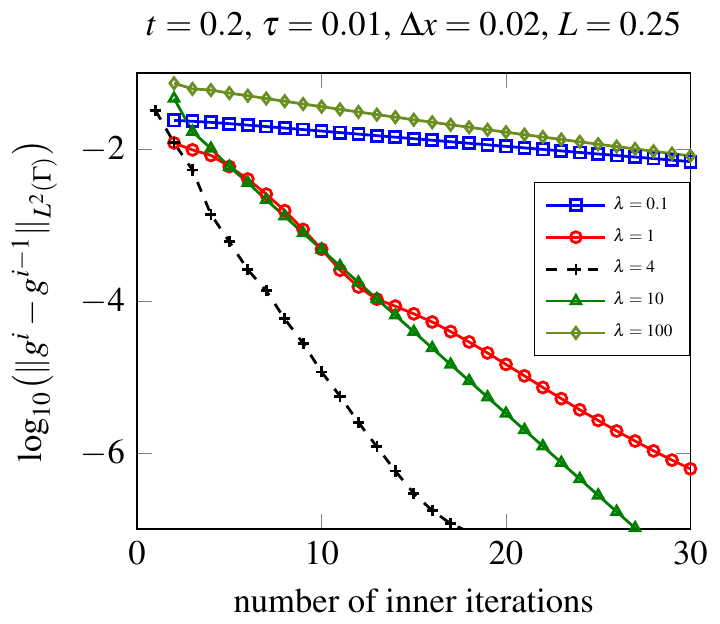}
    \caption{The decay of the $g$-error in terms of $\lambda$.}
    \label{figDDnum.lambda.g}
  \end{subfigure}
  \caption{The influence of $\lambda$ on the convergence rate. The parameters for the LDD scheme are $\tau=.01$, $\Delta x=0.02$,
$L_l=0.25$ at $t=0.2$.}
  \label{figDDnum.lambda}
\end{figure*}
Figure \ref{figDDnum.lambda} shows the influence of the parameter $\lambda$ on error characteristics. All the results
shown
are for the case with exact solution. Figure \ref{figDDnum.lambda.err} focuses on the errors $\norm{p^{n,i} - p^{n,i-1}}_{L^2(\dom)}$ 
\begin{wrapfigure}[12]{r}{0.47\textwidth}%[hbtp]
%    \vspace{-.17cm}
  \centering
  \includegraphics[width=0.47\textwidth]{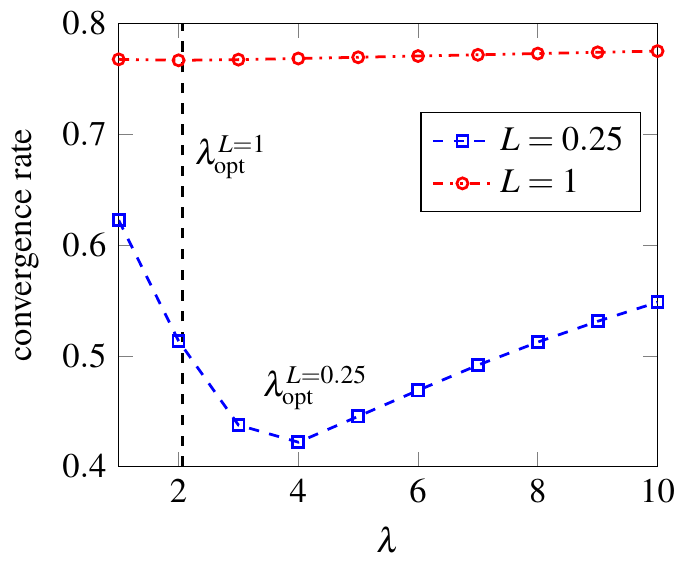}
  \caption{Convergence rate vs $\lambda$ for $L=0.25$ and $L=1$. For $L=0.25$, $\lambda_{\mbox{\scriptsize opt}} \approx 4$.}
  \label{figDDnum.optLambda}
\end{wrapfigure}
% %
%
on the domain $\dom$, while Fig. \ref{figDDnum.lambda.g} depicts the $L^2$-errors $\norm{g^i -
g^{i-1}}_{L^2(\Gamma)}$ on the interface for the same time step. Clearly, $\lambda$ has tremendous impact on the
convergence rate. The convergence rate rapidly increases with $\lambda$ at first but after a certain point the
convergence rate starts decreasing. This trend is noticeable in both plots of Figure \ref{figDDnum.lambda}. This
indicates
that there is an optimal lambda $\lambda_{\mbox{\scriptsize opt}}$ for which the whole scheme has a fastest
convergence rate. The optimality  of $\lambda$ is actually a well studied behaviour in  the domain decomposition
literature. In \cite{LiZhen2008,QinXu2006} it has been shown that $\lambda_{opt}$  depends at least on mesh size and
sub-
domain size.
Later we will show that it also depends on $L_l$ and $\tau$ in our case.
This control over the convergence rate is the reason why the $\lambda$-formulation was chosen over the
convex-combination formulation given in Remark \ref{RemarkFormulationVariants}. To illustrate this, Fig.
\ref{figDDnum.eta} shows
the same plots as Figure \ref{figDDnum.lambda} but for the convex-combination formulation. In order to differentiate between plots
more easily, we use the combined formulation (\ref{strongiterschemeeqn2''}), (\ref{stronggliupdate''}) and set $M=1$. For
$\eta =0.01$
the convex-combination formulation even fails to converge. In all other cases the convergence is considerably slower.

\begin{figure*}[tp]
\begin{subfigure}[b]{.49\textwidth}
  \centering
  \includegraphics[width=\textwidth]{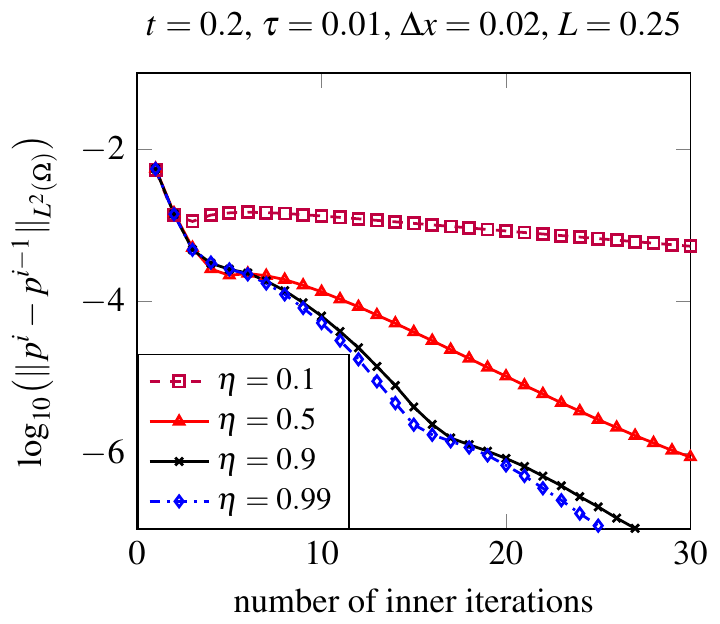}
  \caption{The decay of the pressure error in terms of $\eta.$}
  \label{figDDnum.eta.err}
\end{subfigure}
\begin{subfigure}[b]{.49\textwidth}
  \centering
  \includegraphics[width=\textwidth]{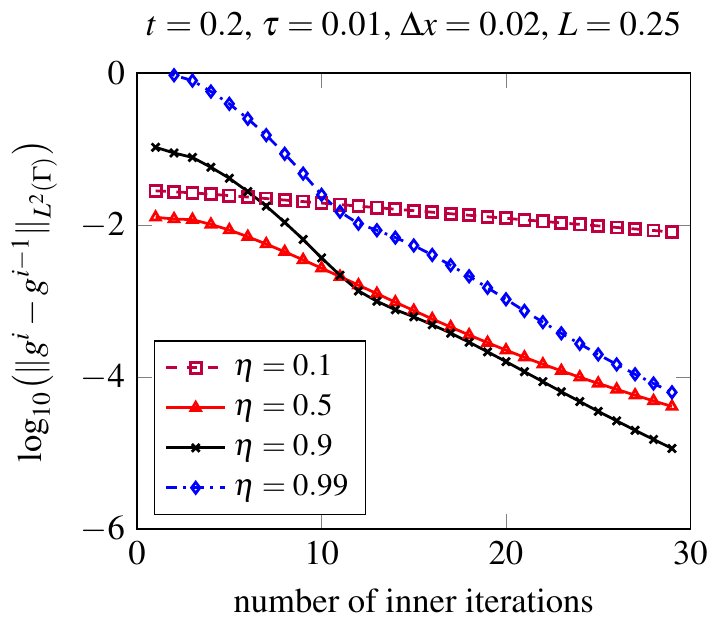}
  \caption{The decay of the $g$-error in terms of $\eta.$}
  \label{figDDnum.eta.g}
\end{subfigure}
  \caption{The influence of $\eta$ on the convergence rate in the convex-combination formulation ($M=1$ in Remark
  \ref{RemarkFormulationVariants}).
  The parameters for the LDD scheme are $\tau=0.01$, $\Delta x=0.02$, $L_l=0.25$ at $t=0.2$.}
  \label{figDDnum.eta}
\end{figure*}
\subsubsection*{The influence of \texorpdfstring{$L_l$}{Ll}}
\begin{wrapfigure}[12]{R}{0.48\textwidth}%[hbtp]
  \vspace{-0.58cm}
  \centering
  \includegraphics[width=.48\textwidth]{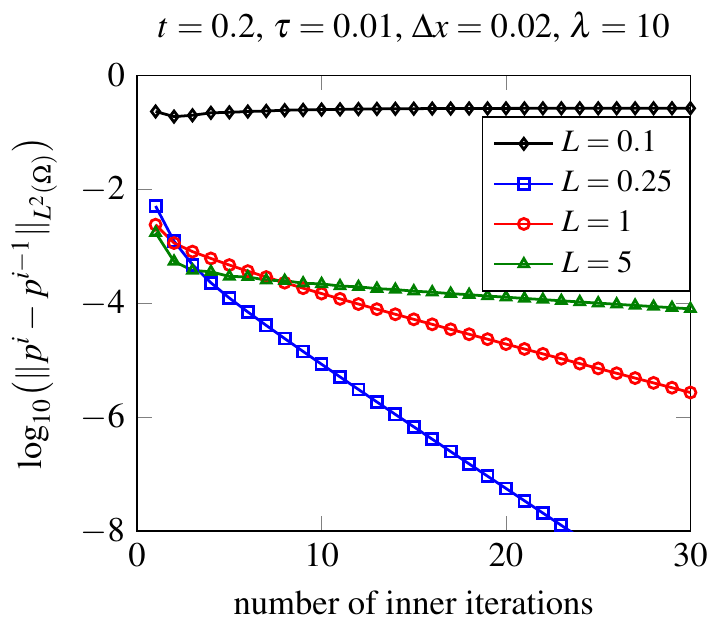}
  \caption{The influence of $L$ on the convergence rate, as obtained for the inner iterations for the $50^{th}$ time step. }
  \label{figDDnum.L}
% \end{figure}
\end{wrapfigure}
We briefly give an overview over the influence of $L_l$ on the convergence rate. Figure \ref{figDDnum.L} depicts this for
$L:=L_1=L_2$. For L-schemes it is common to diverge if $L$  is too small, which seems to be
the case for $L=0.1$. On the other hand, the convergence rate decreases significantly for very large $L$, a behaviour
that is a common trait of L-schemes as well, cf. \cite{Radu2004}. It is best to choose $L$ as small as possible, yet great
enough to ensure convergence of the scheme.  Note that $L_l=0$ represents the original (nonmodified) Picard iteration case and
Figure \ref{figDDnum.L} suggests that the original Picard scheme fails for these problems.

\subsubsection*{The dependence of \texorpdfstring{$\lambda_{\mbox{\scriptsize opt}}$}{lambdaopt} on
\texorpdfstring{$L_l$}{Ll}, \texorpdfstring{$\tau$}{tau} and \texorpdfstring{$\Delta x$}{Deltax}}
In this last section we investigate numerically how $\lambda_{\mbox{\scriptsize opt}}$
depends on the choice of $L$, $\tau$ and $\Delta x$. For a fixed grid in time and space
Table \ref{InfluenceOfParameters1} lists convergence rates for different $\lambda$ and $L$.
With this table we can guess the interval in which $\lambda_{\mbox{\scriptsize opt}}$
lies. Within this estimated interval, Fig. \ref{figDDnum.optLambda} shows how the
convergence rate varies with $\lambda$ for fixed $L$. For $L=0.25$, $\Delta x=0.02$ and
$\tau=0.01$ the fastest convergence is achieved for $\lambda=4$ (this is why $\lambda=4$
was chosen for the above comparisons, wherever the specified $L$, $\Delta x$, $\tau$ set was used).
The $\lambda$ dependence for higher values of $L$ is less pronounced.

% \begin{figure}[bthp]
For a fixed $L$, Tables \ref{InfluenceOfParameters2a}, \ref{InfluenceOfParameters2b} show the variance of
$\lambda_{\mbox{\scriptsize opt}}$ with respect to time-step and mesh size respectively.  The shown tables are of
course only a rough estimate of $\lambda_{\mbox{\scriptsize opt}}$. 
Due to computation time, it is a tedious process to
find a close to exact value of $\lambda_{\mbox{\scriptsize opt}}$, especially for very small time-step sizes. In
practice the values are numerically guessed.
The results indicate quite a strong correlation of $\lambda_{\mbox{\scriptsize opt}}$ with the time-step size,
contrasted by a rather minor correlation with the mesh size.
\begin{table}[htbp]
  \centering
  %%%% LAMBDA DEPENDENCE OF CONVERGENCE RATE
%  \taburowcolors[1]1{black!70..white}
\begin{tabular}{| c | c | c | c | c | c |} %to 0.48\textwidth {X[c,m] X[c,m] X[c,m] X[c,m] X[c,m]}%\hline%
 \hline %&&&&&   \\  %
 %  \begin{tabu} to 0.48\textwidth {X[0.5c,m] X[c,m] X[c,m] X[c,m] X[c,m] X[0.8c,m]}
%	& \multicolumn{4}{c}{\textcolor{white}{Avarage contraction rate}} & \\
%  \rowcolor[gray]{0.7}
  $L$  &\textcolor{black}{$\lambda=0.1$} &\textcolor{black}{$\lambda=1$}
  &\textcolor{black}{$\lambda=10$}   &\textcolor{black}{$\lambda=100$} & $\lambda_{\mbox{\scriptsize opt}} \in $\\
  \hline
%  \taburowcolors[1]1{gray!25..white}
    0.1		& diverged & diverged & diverged & diverged & -         \\
    0.25 	& 0.9020   & 0.6223   & 0.5480   & 0.7721   & ($1$,$10$)\\
    1 		& diverged & 0.7675   & 0.7750   & 0.8138   & ($1$,$10$)\\
    5 		& diverged & 0.8993   & 0.8718   & 0.8708   & ($10$,$100$) \\
%     \tabuphantomline
  \hline
\end{tabular}

  \caption{The dependence of the convergence rates on $\lambda$ and $L$: the geometric average of the contraction rates over the first 20
iterations and for different $(L,\lambda)$ pairs is given in the first columns, whereas the last gives the interval for $\lambda_{\mbox{\scriptsize opt}}$. Here, $\Delta x=0.02$,
$\Delta t=0.01$, $t=0.2$. }\label{InfluenceOfParameters1}
\end{table}
\begin{table}[htbp]
  \begin{subtable}{0.48\textwidth}
%    \rowcolors{2}{gray!25}{white}
    \centering
%    \tabulinesep=1.2mm
%    \taburulecolor{gray!48}
%    \taburowcolors[1]1{black!70 .. black!70}
     \begin{tabular}{| c | c | c | c |} %to 0.48\textwidth {X[c,m] X[c,m] X[c,m] X[c,m] X[c,m]}%\hline%
 \hline %&&&   \\  %
%    \begin{tabu} to \textwidth {X[1.2c,m] X[c,m] X[c,m] X[c,m]}
%       \rowcolor[gray]{0.4}
 %     \multicolumn{4}{c}{\textcolor{white}{Approximate $\lambda_{\mbox{\scriptsize opt}}$}} \\
  %    \taburowcolors[2]2{gray!40 .. white}
  \textcolor{black}{$\Delta t $} &\textcolor{black}{0.1} &\textcolor{black}{0.01} &\textcolor{black}{0.001}\\
  Nr iter.?        &2   &4    &6   \\
  Avg. CR &0.4444 &0.4221 &0.5408 \\
  % \tabuphantomline
  \hline
\end{tabular}

    \caption{$\lambda_{\mbox{\scriptsize opt}}$ for $\Delta x=0.02$, $L_l=0.25$}
    \label{InfluenceOfParameters2a}
  \end{subtable}%\\[3mm]
  \begin{subtable}{0.48\textwidth}
    %\rowcolors{2}{gray!10}{black}
    \centering
%    \tabulinesep=1.2mm
%    \taburowcolors[1]1{black!70 .. black!70}
    \begin{tabular}{| c | c | c | c | c |} %to 0.48\textwidth {X[c,m] X[c,m] X[c,m] X[c,m] X[c,m]}%\hline%
  \hline %&&&&   \\  %
%   \begin{tabu} to \textwidth {X[1.2c,m] X[c,m] X[c,m] X[c,m] X[c,m]}
  %\multicolumn{5}{c}{\textcolor{white}{Approximate $\lambda_{\mbox{\scriptsize opt}}$}} \\
  %\taburowcolors[2]2{gray!40 .. white}
  \textcolor{black}{$\Delta x $} &\textcolor{black}{0.1} &\textcolor{black}{0.05} &\textcolor{black}{0.02}
  &\textcolor{black}{0.01}\\
  Nr iter.?  &3    &4    &4    &4      \\%
  Avg. CR   &0.4398 &0.4270 &0.4221 &0.4221 \\
  %\tabuphantomline
  \hline
\end{tabular}

    \caption{$\lambda_{\mbox{\scriptsize opt}}$ for $\Delta t=0.01$, $L_l=0.25$}
    \label{InfluenceOfParameters2b}
  \end{subtable}
  \caption{The dependence of $\lambda_{\mbox{\scriptsize opt}}$ on $\Delta t$ and on $\Delta x$.}
  \label{InfluenceOfParameters2}
\end{table} 

%
% %

\section{Conclusion}
We considered a nonlinear parabolic problem appearing as mathematical model for variably saturated flow in porous media. For the numerical solution of the nonlinear, time discrete problems we proposed a combined scheme ($LDD$) that is based on a fixed point iteration (the $L$-scheme), and on a domain decomposition scheme involving Robin type coupling conditions at the interface separating different subdomains. The result is a scheme featuring the advantages of both approaches: an unconditional convergence, regardless of  time step and starting point, as well as a decoupling of the time discrete problems into subproblems that can be solved in parallel. The stability, robustness and efficiency of the method is tested for various cases and also compared to Newton and Picard schemes. The tests include situations where the latter diverge whereas the proposed scheme is converging. In summary, the key advantages of the method are:

\begin{itemize}
  \item The LDD scheme converges unconditionally. It can provide accurate results even in situations where the Picard or Newton iterations fail.
  \item In conjunction with a suitable space discretisation, it provides a decoupled, mass conservative approach. This is very useful in particular when dealing with
  models defined in media with block-type heterogeneities, where the material properties in different blocks may vary significantly.
    \item Though linearly convergent, the computational time required by the LDD scheme for achieving a certain accuracy of the approximation is comparable to the time needed by Newton and Picard schemes, and much faster than a standard L-scheme applied to the model in the entire domain. This efficiency is due to the fact that the scheme needs less time per inner iteration than a scheme defined in the entire domain. Moreover LDD is parallelisable, which gives the possibility of increasing its efficiency even further.
  \item The convergence rate of LDD schemes depends on the choice of $L$ and $\lambda$. With the optimal choice of parameters, the convergence order can be reduced significantly.
\end{itemize}

% \begin{acknowledgements*}
\section*{Acknowledgements}
{\itshape This work was supported by Netherlands Organisation for Scientific Research NWO Visitors Grant 040.11.499, Odysseus programme of the Research Foundation - Flanders FWO (G0G1316N), by UHasselt Special Research Fund project BOF17BL04, the Norwegian Research Council through the projects NRC 255510 (CHI) and NRC 255426 (IMMENS) as well as Shell-NWO/FOM CSER programme (project 14CSER016) and the German Research Foundation through IRTG 1398 NUPUS (project B17).}
% \end{acknowledgements*}
% \section*{References}
\bibliography{articlebibliography}   % name your BibTeX data base
\end{document}